\documentclass[12,a4paper]{amsart}
\usepackage{amscd,amssymb,array,amsmath,epsfig,booktabs}
\usepackage{graphicx,color}
\usepackage{ulem}

\def\limsup{\mathop{\overline{\rm lim}}}

\newtheoremstyle{sty}%
{}%
{}%
{\sl}%
{}%
{\bf}%
{}%
{.5em}%
{}%
\theoremstyle{sty}
\newcommand{\D}{\displaystyle}
\newcommand{\e}{{\rm e}}
\newcommand{\hfir}{H_{\nu}^{(1)}}
\newcommand{\hsec}{H_{\nu}^{(2)}}
\newcommand{\lommel}{S_{\mu,\,\nu}}
\newcommand{\sgn}{{\rm sgn}\,}
\newtheorem{theorem}{Theorem}[section]
\newtheorem{corollary}[theorem]{Corollary}
\newtheorem{remark}[theorem]{Remark}
\newtheorem{lemma}[theorem]{Lemma}
\newtheorem{proposition}[theorem]{Proposition}
\numberwithin{equation}{section}

\renewcommand{\Re}{{\rm Re}\,}
\renewcommand{\Im}{{\rm Im}\,}

\allowdisplaybreaks[4]

\begin{document}

\title[Complex oscillation, function-theoretic quantization of ODEs]{On complex oscillation, {function-theoretic} quantization of  non-homogenous {periodic} ODEs and special functions}

\date{$22^{\textrm{nd}}$ May, 2011}

\author[Y. M. Chiang]{Yik-Man Chiang}

\address{Department of Mathematics, The Hong Kong
University of Science \& Technology, Clear Water Bay, Kowloon, Hong
Kong, China} \email{machiang@ust.hk}

\author[K. W. Yu]{Kit-Wing Yu}
\address{Department of Mathematics, The Hong Kong
University of Science \& Technology, Clear Water Bay, Kowloon,
Hong Kong, China} \email{kitwing@hotmail.com}

\dedicatory{Dedicated to the {seventieth} birthday of Lo Yang}

\begin{abstract}
New necessary and sufficient conditions are given for the quantization of a class of periodic second order non-homogeneous ordinary differential equations in the complex plane in this paper. The problem is studied from the viewpoint of complex oscillation theory first developed by Bank and Laine (1982, 1983) {and Gundersen and Steinbart (1994)}.  We show that when a solution is complex non-oscillatory (finite exponent of convergence of zeros) then the solution, which can be written as special functions, must degenerate. This gives a necessary and sufficient condition when the Lommel function has finitely many zeros in every branch and this is a type of quantization for the non-homogeneous differential equation. The degenerate solutions are of polynomial/rational-type functions which are of independent interest. In particular, this shows that complex non-oscillatory solutions of this class of differential equations are equivalent to the subnormal solutions considered in a previous paper of the authors (to appear). In addition to the asymptotics of special functions, the other main idea we apply in our proof is a classical result of E. M. Wright which gives precise asymptotic locations of large zeros of a functional equation.
\end{abstract}



\maketitle

\section{\bf Introduction and the main results}

Let $A(z)$ be a transcendental entire function, and that $f(z)$ an
entire function solution of the differential equation

    \begin{equation}\label{E:homogeneous-de}
        f^{\prime\prime}+A(z)f=0.
    \end{equation}\smallskip

We use the $\sigma(f)$ and $\lambda(f)$ to denote the {\it order} and \textit{exponent of convergence of zeros},  respectively, of an entire function $f(z)$. A solution is called \textit{complex oscillatory} if $\lambda(f)=+\infty$ and \textit{complex non-oscillatory} if $\lambda(f)<+\infty$. Reader can refer to \cite{Hayman:1975} or \cite{Laine:1993} for the notation and background related to Nevanlinna theory where this research originated. However, we shall not make use of the Nevanlinna theory in the rest of this paper. For earlier treatments of various complex oscillation problems considered, we refer the reader to,  for examples, \cite{Bank:Laine:1982, Bank:Laine:1983, Bank:Laine:Langley:1986, Bank:Laine:Langley:1989, Chiang:Ismail:2006, Chiang:Ismail:2010, Gundersen:Steinbart:1994, Huang:Chen:Li:2009, Shimomura:2002}.
 In this paper, we consider the complex oscillation problem of a class of non-homogeneous differential equations which includes the simple looking equation

\begin{equation}\label{E:Chiang-Yu-equation-speical-case}
f''+(\e^{2z}-\nu^2)f=\sigma \e^{(\mu+1)z},
\end{equation}\smallskip

\noindent as a special case, where $\mu,\,\nu,\,\sigma \in \mathbf{C}$ and $\sigma \neq 0$. It is well-known that {all solutions of} the equation {are entire}, see \cite{Laine:1993}. \smallskip

In \cite[Theorem 1.2]{Chiang:Yu:2008}, the authors gave several explicit solutions of (\ref{E:Chiang-Yu-equation-speical-case}) in terms of the sum of the Bessel functions of first and second kinds $J_\nu(\zeta)$ and $Y_\nu(\zeta)$, and the \textit{Lommel function} $\lommel(\zeta)$ (\S \ref{A:Bessel}). In fact, the general solution of (\ref{E:Chiang-Yu-equation-speical-case}) can be written as

\begin{equation}\label{E:Chiang-Yu-equation-speical-case-solutions}
f(z)=AJ_\nu(\e^z)+BY_\nu(\e^z)+\sigma \lommel(\e^z),
\end{equation}\smallskip

\noindent where $A,\,B \in \mathbf{C}$ and $\lommel(\zeta)$ is a \textit{particular integral} of the non-homogeneous Bessel differential equation

\begin{equation}\label{E:non-homogeneous-bessel-de}
\zeta^2 y''(\zeta)+\zeta y'(\zeta)+(\zeta^2-\nu^2)y(\zeta)=\zeta^{\mu+1}.
\end{equation}\smallskip

\noindent The functions $J_\nu(\zeta)$ and $Y_\nu(\zeta)$ are two \textit{linearly independent} solutions of the corresponding homogeneous Bessel differential equation of (\ref{E:non-homogeneous-bessel-de}). \smallskip

The Lommel function $S_{\mu,\,\nu}(\zeta)$, is a special function that plays important roles in numerous physical applications (see e. g. \cite{Lyth:Stewart:1996, Mobilia:Bares:2001, Walker:1904}), was first appeared to be studied by Lommel in \cite{Lo1876}. We refer the reader to \cite{Chiang:Yu:2008} and the references therein for further discussion about the background and the applications of the Lommel function.\smallskip

The authors' previous work \cite[Theorem 1.2]{Chiang:Yu:2008} concerns the \textit{subnormality} of the solutions of (\ref{E:Chiang-Yu-equation-speical-case}). We recall that an entire function $f(z)$ is called \textit{subnormal} if either

\begin{align}\label{D:subnormal-definition}
\limsup_{r \to +\infty}\frac{\log\log M(r,\,f)}{r}=0\quad \mbox{or} \quad \limsup_{r \to +\infty}\frac{\log T(r,\,f)}{r}=0
\end{align}\smallskip

\noindent holds, where $\D M(r,\, f)=\max_{|z|\le r}|f(z)|$ denotes the usual maximum modulus of the entire function $f(z)$ and $T(r,\, f)$ is the Nevanlinna characteristics of $f(z)$. We have shown that solutions of (\ref{E:Chiang-Yu-equation-speical-case}) are subnormal, that is, if (\ref{D:subnormal-definition}) holds, if and only if $(A,\,B)=(0,\,0)$ and either $\mu+\nu=2p+1$ or $\mu-\nu=2p+1$ holds for a non-negative integer $p$ in (\ref{E:Chiang-Yu-equation-speical-case-solutions}) and the subnormal solutions have the form {given by the formulae} (\ref{E:lommel-sol-special}) and (\ref{E:ck-special}). In other words, \textit{subnormal solutions and finite order solutions of {\rm (\ref{E:Chiang-Yu-equation-speical-case})} are equivalent}, see Corollary \ref{R:Equivalent-conditions} below. This provides a new \textit{non-homogeneous {function-theoretic} quantization-type} result for the equation (\ref{E:Chiang-Yu-equation-speical-case}) {whose explanation will be given in \S 8}. \smallskip

\begin{remark}\label{R:Chiang-Yu-results}\rm The authors also generalized the above results to a much more general equation (\ref{E:Chiang-Yu}) below (\cite[Theorem 1.4]{Chiang:Yu:2008}), and a number of interesting corollaries. For examples, orders of growth of the entire functions $S_{\mu,\,\nu}(\e^z)$ and ${\bf H}_\nu(\e^z)$ were determined and \textit{non-homogeneous {function-theoretic} quantization-type results} were also obtained. See \cite[Theorem 1.7, \S 6]{Chiang:Yu:2008} for details.
\end{remark}\smallskip

For homogeneous differential equation (\ref{E:homogeneous-de}), it is known that a solution $f(z)$ could have $\lambda(f)<+\infty$ but its growth being \textit{not subnormal}, that is,\\ $\D \limsup_{r \to +\infty}{\log\log M(r,\,f)}/{r}=+\infty$. Thus one may ask what is the relationship between subnormality of solutions and its exponent of convergence of zeros of (\ref{E:Chiang-Yu-equation-speical-case}). We show that the finiteness of the two measures are equivalent.\smallskip

Here are our main results:

\begin{theorem}\label{T:Chiang-Yu-Special} Let $f(z)$ be a solution of {\rm (\ref{E:Chiang-Yu-equation-speical-case})}. Then $\lambda(f)<+\infty$ if and only if $A=B=0$ in {\rm (\ref{E:Chiang-Yu-equation-speical-case-solutions})} and either $\mu+\nu=2p+1$ or $\mu-\nu=2p+1$ for a non-negative integer $p$  and

\begin{align}\label{E:lommel-sol-special}
S_{\mu,\,\nu}(\zeta)=\zeta^{\mu-1}\Bigg[\sum_{k=0}^p \frac{(-1)^kc_k}{\zeta^{2k}}\Bigg],
\end{align}\smallskip

\noindent where the coefficients $c_k,\,k=0,\,1,\ldots,\,p$, are defined by

\begin{align}\label{E:ck-special}
c_0=1\quad \mbox{and}\quad c_k=\prod_{m=1}^k [(\mu-2m+1)^2-\nu^2].
\end{align}
\end{theorem}\smallskip

The above result is a special case of the following Theorem {\ref{T:Chiang-Yu}}. We first introduce a set of more general coefficients. Suppose that $n$ is a positive integer and $A,\, B,\, L,\,M,\, N,\, \sigma,\, \sigma_i,\, \mu_j,\, \nu$ are complex numbers such that $L,\, M$ are non-zero and at least one of  $\sigma_j$, $j \in \{1,\, 2,\ldots,\,n\}$, being non-zero.

    \begin{theorem}\label{T:Chiang-Yu} Let $f(z)$ be an entire solution to the differential equation

        \begin{align} \label{E:Chiang-Yu}
        f''+2Nf'+\big[L^2M^2\e^{2Mz} &+(N^2 -\nu^2 M^2)\big]f \notag \\
        &=\sum_{j=1}^n \sigma_j L^{\mu_j+1}M^2\e^{[M(\mu_j+1)-N]z}.
        \end{align}\smallskip

    \noindent Then $f(z)$ is given by

    \begin{align}\label{E:f-solution}
    f(z)={\rm e}^{-Nz} \bigg[AJ_\nu(L{\rm e}^{Mz})+BY_\nu(L{\rm e}^{Mz})+\sum_{j=1}^n \sigma_j S_{\mu_j,\, \nu}(L{\rm e}^{Mz})\bigg].
    \end{align}\smallskip

    \noindent Moreover, suppose that all the $\Re(\mu_j)$ are distinct. Then we have $\lambda(f)<+\infty$ if and only if $A=B=0$ and for each non-zero $\sigma_j$, we have either

    \begin{equation}\label{E:conditions}
    \mu_j + \nu=2p_j+1 \quad \mbox{or}\quad \mu_j-\nu=2p_j+1,
    \end{equation}\smallskip

    \noindent where $p_j$ is a non-negative integer and

        \begin{equation}\label{E:lommel-sol}
        S_{\mu_j,\, \nu}(\zeta)=\zeta^{\mu_j-1}\left[\sum_{k=0}^{p_j}
        \frac{(-1)^k c_{k,\, j}}{\zeta^{2k}}\right],
        \end{equation}\smallskip

    \noindent where $c_{k,\,j}\,(k=0,\,1,\ldots,\,p-1;\, j=1,\,2,\ldots,\,n)$ are defined by

    \begin{equation}\label{E:ck}
    c_{0,\,j}=1\quad \mbox{and}\quad c_{k,\,j}=\prod_{m=1}^k [(\mu_j-2m+1)^2-\nu^2].
    \end{equation}
\smallskip
    \end{theorem}\smallskip

As an immediate consequence of the Theorems \ref{T:Chiang-Yu-Special} or \ref{T:Chiang-Yu}, we get the following result which gives us information about the number of zeros of the Lommel function $S_{\mu,\,\nu}(\zeta)$ in the sense of Nevanlinna's value distribution theory:
\smallskip

\begin{corollary}\label{T:Chiang-Yu-Lommel-Zeros} Suppose that $\sigma_j$ are complex constants such that at least one of $\sigma_j$ is non-zero, where $j=1,\,2,\ldots,\,n$ . Suppose further that $\Re(\mu_j)$ are distinct and $S_{\mu_j,\,\nu}(\zeta)$ are Lommel functions of arbitrary branches given in Lemma {\rm \ref{L:lommel-continuation-general-1}}. Then each branch of the function

\begin{equation}\label{E:F-definition}
F(\zeta)=\sum_{j=1}^n \sigma_j S_{\mu_j,\,\nu}(\zeta)
\end{equation}\smallskip

\noindent has finitely many zeros if and only if either $\mu_j+\nu=2p_j+1$ or $\mu_j-\nu=2p_j+1$ for non-negative integers $p_j$. In particular, the special case $n=1$ implies that each branch of $S_{\mu,\, \nu}(\zeta)$ has finitely many zeros must satisfy either $\mu+\nu=2p+1$ or $\mu-\nu=2p+1$ for a non-negative integer $p$.
\end{corollary}
\smallskip

The Lommel function $S_{\mu,\, \nu}(\zeta)$, as in the cases of many classical special functions, has, in general, infinitely many branches, that is, its covering manifold has infinitely many sheets. The values of the function in different branches are given by so-called \textit{analytic continuation formulae}. Such analytic continuation formulae in its full generality, first derived by the authors in \cite{Chiang:Yu:2008}, are given in Lemma \ref{L:lommel-continuation-general-1}.\smallskip

\begin{corollary}\label{R:Equivalent-conditions}  Suppose that $f(z)$ is a solution of {\rm (\ref{E:Chiang-Yu-equation-speical-case})}. Then we have $\lambda(f)<+\infty$ if and only if the solution $f(z)$ is subnormal.
\end{corollary}\smallskip

\begin{remark}\label{R:Bessel-lommel-independence-branches} \rm We note that for all values of $\mu_j$ and $\nu$, $J_\nu(L{\rm e}^{Mz}),\,Y_\nu(L{\rm e}^{Mz})$ and $S_{\mu_j,\,\nu}(L{\rm e}^{Mz})$ are entire functions in the complex $z$-plane. Hence they are single-valued functions and so are \textit{independent of the branches} of $S_{\mu_j,\,\nu}(\zeta)$.
\end{remark}\smallskip

The main idea of our argument in the proofs is based on the asymptotic expansions of special functions (Bessel and Lommel functions), the analytic continuation formulae for $\lommel(\zeta)$ (Those formulae were first discovered by the authors which play a very important role in \cite{Chiang:Yu:2008}  and also in this paper), the asymptotic locations of the zero of the {transcendental} equation $z \e^z=a$ given by Wright \cite{Wright:1959:1}, \cite{Wright:1959:2} and application of Rouch\'{e}'s theorem on suitably chosen contours in the complex plane.
\medskip

This paper is organized as follows. We introduce the Lommel transformation in \S 2, which serves as a crucial step in our proof to transform the equation (\ref{E:Chiang-Yu}) into the equation

\begin{equation}\label{E:generalized-Lommel}
\zeta^2y''(\zeta)+\zeta y'(\zeta)+(\zeta^2-\nu^2)y(\zeta)=\sum_{j=1}^n\sigma_j
\zeta^{\mu_j+1}.
\end{equation}\smallskip

\noindent Since we need to consider the \textit{different branches} of {the function (\ref{E:general-soln-first-second})} in the proof of the Theorem \ref{T:Chiang-Yu}, so the analytic continuation formulae for $\lommel(\zeta)$ come into play at this stage. We quote these formulae, which were derived in \cite{Chiang:Yu:2008}, in \S 3 for easy reference. Besides, we need information about the zeros of the function $g(\zeta)=\widehat{C}\e^{i\zeta}+\widehat{\sigma}\zeta^{\mu-\frac{1}{2}}$, where $\widehat{C} \neq 0$ and $\mu \neq \frac{1}{2}$. It turns out that Wright has already investigated the precise locations of zeros of the equation $z\e^z=a$ in \cite{Wright:1959:1,Wright:1959:2}, where $a \neq 0$. In fact, this problem is of considerable scientific interest, see for examples  (\cite{Bellman:Cooke:1963, Wright:1955}). Since we need to modify Wright's method in the preliminary construction of one of the contours used in the proof of the Theorem \ref{T:Chiang-Yu}, so we shall sketch Wright's method in \S 4. A detailed study of the zeros of the function $g(\zeta)=\widehat{C}\e^{i\zeta}+\widehat{\sigma}\zeta^{\mu-\frac{1}{2}}$ will be given in \S 5, followed by the proof of the Theorem \ref{T:Chiang-Yu} in \S 6. A proof of the Corollary \ref{T:Chiang-Yu-Lommel-Zeros} is presented in \S 7 and a discussion about the non-homogeneous {function-theoretic} quantization-type result will be given in \S 8. Appendix A contains all the necessary knowledge about the Bessel functions and the Lommel functions that are used in this paper.\smallskip

\section{\bf The Lommel transformations}\label{S:lomme-transformations} Lommel investigated transformations that involve Bessel equations \cite{Lommel:1868} in 1868.\footnote{We mentioned that the same transformations were also considered independently by Pearson \cite[p. 98]{Watson:1944} in 1880.} Our standard references are \cite[4.31]{Watson:1944}, \cite[p. 13]{Erdelyi:Magnu:Oberhettinger:Tricom:1953} and \cite[\S 2]{Chiang:Yu:2008}.  Lommel considered the transformation $\zeta=\alpha x^\beta$ and $y(\zeta)=x^{\gamma}u(x)$, where $x$ and $u(x)$ are the new independent and dependent variables respectively, $\alpha, \beta \in \mathbf{C}\setminus\{0\}$ and $\gamma \in \mathbf{C}$. We apply this transformation to equation (\ref{E:generalized-Lommel}) to obtain a second order differential equation in $u(x)$ whose general solution is $x^{-\gamma} y(\alpha x^\beta)$. Following the idea in \cite{Chiang:Yu:2008}, we apply a further change of variable by
$x={\rm e}^{z}$ and $f(z)=u(x)$ to the said differential equation in $u(x)$ and then replacing $\alpha,\beta$ and $\gamma$ by $L,\, M$ and $N$ respectively. This process yields (\ref{E:Chiang-Yu}). As we have noted in \S 1 that the general solution of (\ref{E:non-homogeneous-bessel-de}) is given by a combination of the Bessel functions of first and second kinds and the Lommel function $S_{\mu,\,\nu}(\zeta)$ (see \cite[7.7.5]{Erdelyi:Magnu:Oberhettinger:Tricom:1953}), hence the general solution to (\ref{E:generalized-Lommel}) is

\begin{align}
y(\zeta)&=A J_\nu (\zeta)+B Y_\nu (\zeta)+\sum_{j=1}^n \sigma_j S_{\mu_j,\, \nu}(\zeta) \label{E:general-soln-first-second}\\
&=C H_\nu^{(1)}(\zeta)+DH_\nu^{(2)}(\zeta)+\sum_{j=1}^n \sigma_j S_{\mu_j,\, \nu}(\zeta),\label{E:general-soln-hankel}
\end{align}\smallskip

\noindent where $C=\frac{1}{2}(A-iB)$ and $D=\frac{1}{2}(A+iB)$. It is easily seen that $A=B=0$ if and only if $C=D=0$. Thus the general solution $f(z)={\rm e}^{-N z}y(L {\rm e}^{M z})$ of {(\ref{E:Chiang-Yu}}) assumes the form

    \begin{align}\label{E:f-solution-hankel}
    f(z)={\rm e}^{-Nz} \bigg[CH_\nu^{(1)}(L{\rm e}^{Mz})+DH_\nu^{(2)}(L{\rm e}^{Mz})+\sum_{j=1}^n \sigma_j S_{\mu_j,\, \nu}(L{\rm e}^{Mz})\bigg].
    \end{align}\smallskip

\section{\bf Analytic continuation formulae for the Lommel function}

We first note that the Lommel functions $\lommel(\zeta)$ have a rather complicated
definition with respect to different subscripts $\mu$ and $\nu$ (in four different cases) even in the principal branch. In this section, we shall not repeat the description of its definition, interested readers please refer to \cite[\S 3.1]{Chiang:Yu:2008} and the references therein. Here we only record the analytic continuation formulae of $S_{\mu,\,\nu}(\zeta)$ and the proofs of them can be found in \cite[\S 3.2-3.5]{Chiang:Yu:2008}. The asymptotic expansion and the linear independence property of $\lommel(\zeta)$ will be given in the appendix.\smallskip

Let $\chi_\pm :=\frac{1}{2}(\mu \pm \nu+1)$. We define the constants

\begin{equation}\label{E:Ks-definition}
\begin{split}
K&:=2^{\mu-1}\Gamma(\chi_+)\Gamma(\chi_-),\, K_+:=Ki[1+\e^{(-\mu+\nu)\pi i}]\cos\big(\frac{\mu+\nu}{2}\pi\big),\\
K'_\pm &:=\pi 2^{\nu-2} i\e^{-m\nu\pi i}\Gamma(\nu)\big[U_{m-1}(\cos \nu\pi)\e^{(m \pm 1)\nu\pi i}-m\big],\\
K''_\pm &:=-\frac{m\pi^2(m \pm 1)}{4},
\end{split}
\end{equation}\smallskip

\noindent where $\Gamma$ is the gamma function and $\D U_m(\cos \nu\pi):=\frac{\sin m\nu\pi}{\sin \nu\pi}$ is the \textit{Chebyschev polynomials of the second kind}.\smallskip

When $\mu\pm \nu\not= 2p+1$ for any integer $p$, then we have

\begin{lemma}\cite[Theorem 3.4]{Chiang:Yu:2008}\label{L:lommel-continuation-general-1} Let $m$ be an integer.

\begin{enumerate}
  \item[(a)] We have

\begin{align}\label{E:lommel-continuation-general}
 S_{\mu,\,\nu}(\zeta \e^{-m\pi i})&=(-1)^m \e^{-m\mu\pi i}
 S_{\mu,\,\nu}(\zeta)+K_+\big[P_m(\cos \nu\pi,\, \e^{-\mu \pi i})H_\nu^{(1)}(\zeta)\\
 &\quad+\e^{-\nu\pi i }P_{m-1}(\cos \nu\pi,\, \e^{-\mu \pi i})H_\nu^{(2)}(\zeta)\big],\notag
\end{align}\smallskip

\noindent where $P_m(\cos \nu\pi,\, \e^{-\mu \pi i})$ is a rational function of $\cos \nu\pi$ and $\e^{-\mu\pi i}$ given by

\begin{align}\label{E:coefficient-P}
P_m(\cos \nu\pi&,\,\e^{-\mu\pi i}) \notag\\
&=\frac{U_{m-1}(\cos \nu\pi) +\e^{-\mu\pi i}U_m(\cos \nu\pi)+(-1)^{m+1}\e^{-(m+1)\mu\pi i}}{[1+\e^{-(\mu+\nu)\pi i}][1+\e^{-(\mu-\nu)\pi i}]}.
\end{align}\smallskip

  \item[(b)]  Furthermore, the coefficients $P_m(\cos \nu\pi,\, \e^{-\mu \pi i})$ and $P_{m-1}(\cos \nu\pi,\, \e^{-\mu \pi i})$ are not identically zero \textit{simultaneously} for all $\mu,\nu$ and all non-zero integers $m$.
\end{enumerate}
\end{lemma}\smallskip

When either $\mu+\nu$ or $\mu-\nu$ is an odd negative integer $-2p-1$, where $p$ is a non-negative integer, then we have another sets of analytic continuation formulae which are given by

\begin{lemma}\cite[Lemmae 3.6, 3.8, 3.10]{Chiang:Yu:2008}\label{L:lommel-continuation-general-2} Let $m$ be an integer. Then we have
\begin{enumerate}
  \item[(a)] If $-\nu \not\in \{0,\, 1,\, 2,\ldots\}$, then we have

  \begin{align*}
  S_{\nu-2p-1,\,\nu}(\zeta\e^{-m\pi i})&=\e^{-m\nu \pi i}S_{\nu-2p-1,\,\nu}(\zeta)\\
  &\quad +\frac{(-1)^p}{2^{2p}p!(1-\nu)_p}[K_+'\hfir(\zeta)+K_-'\hsec(\zeta)].
  \end{align*}

  \item[(b)] If $\nu=0$, then we have

  \begin{align*}
  S_{-2p-1,\,0}(\zeta \e^{-m\pi i})=S_{-2p-1,\,0}(\zeta)+\frac{(-1)^p}{2^{2p}(p!)^2}[K_+''H^{(1)}_0(\zeta)+K_-''H^{(2)}_0(\zeta)].
  \end{align*}

  \item[(c)] We define $\delta_m=1+(-1)^{m-1}$ and for every polynomial $P_n(\zeta)$ of degree $n$, we define $\widehat{P}_n(\zeta)$ to be the polynomial containing the term of $P_n(\zeta)$ with odd powers in $\zeta$ and $\overline{P}_n(\zeta):=P_n(\zeta)-\delta_m\widehat{P}_n(\zeta)$. If $\nu=-n$ is a positive integer $n$, then we have

  \begin{align*}
  S_{-n-2p-1,\,-n}&(\zeta {\rm e}^{-m\pi i})\\
  &=(-1)^{mn}S_{-n-2p-1,\,-n}(\zeta)+\frac{(-1)^{(m+1)n+p}}{2^{2p+n} n!(p!)^2(1+n)_p}\zeta^{-n}\times \notag\\
  &\qquad\Big\{-\delta_m\big[\widehat{A}_n(\zeta)+\widehat{B}_n(\zeta)S_{-1,\,0}(\zeta)
  +\zeta\widehat{C}_n(\zeta)S_{-1,\,0}'(\zeta)\big]\notag \\
  &\quad+\overline{B}_n(\zeta)\big[K_+''H_0^{(1)}(\zeta)+K_-''H_0^{(2)}(\zeta)\big]\notag\\
  &\qquad-\zeta\overline{C}_n(\zeta)\big[K_+''H_1^{(1)}(\zeta)+K_-''H_1^{(2)}(\zeta)\big]\Big\},\notag
  \end{align*}\smallskip

  \noindent where $A_n(\zeta),\, B_n(\zeta)$ and $C_n(\zeta)$ are polynomials in $\zeta$ of degree at most $n$ such that $A_1(\zeta)=B_1(\zeta) \equiv 0,\,C_1(\zeta) \equiv 1$ and when $n \ge 2$, that they satisfy the following recurrence relations:

  \begin{equation}\label{E:recurrence-A_n-B_n-C_n}
  \begin{split}
  A_n(\zeta)&=-2(n-1)A_{n-1}(\zeta)+\zeta A_{n-1}'(\zeta)+C_{n-1}(\zeta),\\
  B_n(\zeta)&=-2(n-1)B_{n-1}(\zeta)+\zeta B_{n-1}'(\zeta)-\zeta^2 C_{n-1}(\zeta),\\
  C_n(\zeta)&=-2(n-1)C_{n-1}(\zeta)+B_{n-1}(\zeta)+\zeta C_{n-1}'(\zeta).
  \end{split}
  \end{equation}
  \end{enumerate}
\end{lemma}\smallskip

\section{\bf Applications of Wright's result}\label{S:location-zeros}

Suppose that $a$ is a non-zero complex number such that $a=A\e^{i\alpha}$, where $-\pi<\alpha \le \pi$ and $A=|a| \neq 0$. In 1959, Wright \cite{Wright:1959:1, Wright:1959:2} obtained precise asymptotic locations of the zeros of the equation
\
\begin{equation}\label{E:wright-equation}
z\e^z=a
\end{equation}\smallskip

\noindent in terms of rapidly convergent series by constructing the Riemann surface of the inverse function of $z+\log z$. This result of Wright is of considerable scientific interest, particularly in the theory and various applications of difference-differential equations, see \cite{Bellman:Cooke:1963, Wright:1955}. For further applications of this equation, please refer to \cite{Caillol:2003, Corless:1996, Valluri:2001}.\smallskip

In this section, we first describe Wright's result in the Lemma \ref{L:wright} below. Next we apply Wrights' result to obtain finer estimates of the real and imaginary parts of the solutions of (\ref{E:wright-equation}) in the Lemmae \ref{L:real-imaginary-bounds} and \ref{L:real-imaginary-bounds-subseq} below which we need to construct a certain contour needed in the proof of the main result in Proposition \ref{L:g-number-zeros}.\smallskip

Suppose that $n$ is an integer. We let $z(n)=x(n)+iy(n)$ be solutions of equation (\ref{E:wright-equation}), where $x(n)$ and $y(n)$ are real, and are given in the following result of Wright:
\smallskip

\begin{lemma}[\cite{Wright:1959:1, Wright:1959:2}] \label{L:wright} Let $a=A\e^{i\alpha}$, where $-\pi < \alpha \le \pi$ and $A=|a| \neq 0$. Let $\sgn(n)$ be the sign of the non-zero integer $n$. We define

\begin{equation}\label{E:wright-H_n-beta_n}
H_n:=2|n|\pi+\sgn(n)\alpha-\frac{\pi}{2},\quad \beta_n:=\log\frac{A}{H_n},
\end{equation}\smallskip

\noindent taking $\log \frac{A}{H_n}$ real. If $|n|$ is sufficiently large such that

\begin{equation}\label{E:condition-n}
\D 2H_n|\beta_n|<(H_n-1)^2,\quad(\log A)^2<\bigg(H_n-\frac{\pi}{2}\bigg)^2+2(1+\log A)\log H_n+1,
\end{equation}\smallskip

\noindent then the solutions of the equation {\rm (\ref{E:wright-equation})} are given by

\begin{equation}\label{E:wright-solutions}
x(n)=(H_n+\eta_n)\tan \eta_n,\quad  y(n)=\sgn(n)(H_n+\eta_n),
\end{equation}\smallskip

\noindent where $\D \eta_n=\sum_{j=0}^{+\infty} (-1)^j Q_{2j+1}(\beta_n)H_n^{-2j-1}$ and $\{Q_m(t)\}$ is the sequence of polynomials defined by

\begin{equation}\label{E:Qm-definition}
Q_1(t):=t,\quad Q_{m+1}(t):=Q_m(t)+m\int_0^t Q_m(s)\, {\rm d}s,
\end{equation}\smallskip

\noindent where $m$ is a positive integer.
\end{lemma}\smallskip

\begin{remark}\label{R:eta-properties} \rm We deduce from the definition (\ref{E:wright-H_n-beta_n}) that $\beta_n < 0$ for all $n$ sufficiently large and $\beta_n \to -\infty$ as $n \to \pm \infty$. We also note that it follows from (\ref{E:Qm-definition}) that $Q_m(t)$ is a polynomial of degree at most $m$. This and the series representation of $\eta_n$ show that $\D \eta_n=O\Big(\frac{\beta_n}{H_n}\Big)$, $\eta_n<0$ and $\eta_n \to 0$ as $n \to \pm \infty$.
\end{remark}

As we have already mentioned in \S 1 that {we} would like to apply Rouch\'{e}'s theorem to suitable contours. To construct one of these contours, it is necessary to derive accurate bounds for $x(n)$ and $y(n)$ from the following Lemma {\ref{L:real-imaginary-bounds}} of Wright. We include the argument leading to the inequalities to familiarize our readers for later applications.

\begin{lemma}\cite[p. 196]{Wright:1959:2}\label{L:real-imaginary-bounds} Suppose that {$H_n$ and $A$ are as defined in {\rm (\ref{E:wright-H_n-beta_n})}. Then the upper and lower bounds for the real and imaginary parts of the solutions to {\rm (\ref{E:wright-equation})} are given, respectively by}

\begin{align}\label{E:real-bounds}
2\log \frac{A}{(2|n|+1)\pi}-1 < x(n) < \log \frac{A}{2(|n|-1)\pi}+1
\end{align}

\noindent and

\begin{equation}\label{E:imaginary-bounds}
\left\{
  \begin{array}{ll}
    \D (2n-1)\pi+\alpha < y(n) < 2n\pi+\alpha, & \hbox{if $n$ is large positively;} \\
    \D 2n\pi+\alpha< y(n) < (2n+1)\pi+\alpha, & \hbox{if $n$ is large negatively.}
  \end{array}
\right.
\end{equation}\smallskip
\end{lemma}

\begin{proof} It is easy to see  that the inequalities (\ref{E:imaginary-bounds}) follow easily from the definitions (\ref{E:wright-H_n-beta_n}), (\ref{E:wright-solutions}) and the properties of $\eta_n$ in the Remark \ref{R:eta-properties} above. For the inequality (\ref{E:real-bounds}) representing the real part of $z(n)$, we deduce from the power series of $\tan \eta_n$ and equation (\ref{E:wright-solutions}) that

\begin{align*}
x(n)=\eta_n(H_n+\eta_n)\frac{\tan \eta_n}{\eta_n}=\eta_n(H_n+\eta_n)\bigg(1+\frac{\eta_n^2}{3}+\cdots\bigg).
\end{align*}\smallskip

\noindent This implies that the inequalities $\eta_n(H_n+\eta_n)(1-\eta_n) \le x(n) \le \eta_n(H_n+\eta_n)$ hold when $n$ is sufficiently large. On the other hand, the Lemma \ref{L:wright} and Remark \ref{R:eta-properties} assert that $\beta_n-\frac{1}{2}<\eta_nH_n<\beta_n+\frac{1}{2}$ when $n$ is sufficiently large. Combining these two inequalities and the fact $0<1-\eta_n<2$, we deduce

\begin{align}\label{E:real-upper-bounds}
x(n) \le \eta_n(H_n+\eta_n)=\eta_n H_n+\eta_n^2 <\beta_n+\frac{1}{2}+\eta_n^2<\beta_n+1<0
\end{align}\smallskip

\noindent and

\begin{align}\label{E:real-lower-bounds}
x(n) \ge \eta_n(H_n+\eta_n)(1-\eta_n) >2\eta_n(H_n+\eta_n)>2\beta_n-1+2\eta_n^2>2\beta_n-1,
\end{align}\smallskip

\noindent since $\eta_n(H_n+\eta_n)<0$. Hence we deduce from the inequalities (\ref{E:real-upper-bounds}) and (\ref{E:real-lower-bounds}) the inequalities

\begin{equation}\label{L:real-imaginary-bounds-inequality-3}
2\beta_n-1 < x(n) < \beta_n+1<0.
\end{equation}\smallskip

\noindent Since $-\pi <\alpha \le \pi$, we must have  ${-\pi<\sgn(n)\, \alpha \le \pi}$ for every non-zero integer $n$. Then the desired inequalities (\ref{E:real-bounds}) follow from this fact, the inequalities (\ref{L:real-imaginary-bounds-inequality-3}) and the definition (\ref{E:wright-H_n-beta_n}). This completes the proof of the lemma.
\end{proof}

\begin{lemma}\label{L:real-imaginary-bounds-subseq} Let $m$ be a fixed positive integer such that $\log A-\log m\pi+1<-3$. We define $d_r:=2m\pi r^2-\alpha-\pi$ for every real $r>0$ and let $n_k=-mk^2$, where $k$ is a sufficiently large positive integer such that $n_k$ satisfies the inequalities {\rm (\ref{E:condition-n})}. Then we have

\begin{align}\label{E:real-imaginary-bounds-subseq-inequalities}
-5\log k < x(n_k) < -2\log k-2, \quad -d_{k+1}< y(n_k) < -d_k.
\end{align}
\end{lemma}

\begin{proof} Since $n_k$ is large and negative, so the inequalities (\ref{E:real-imaginary-bounds-subseq-inequalities}) for $y(n_k)$ follows easily from the second set of inequalities in (\ref{E:imaginary-bounds}). We next observe that both inequalities

\begin{equation*}
\log(2mk^2+1)\pi \le \log mk^2\pi+\log \pi \quad \mbox{and}\quad \log ( 2mk^2-1)\pi \ge \log mk^2\pi-\log \pi
\end{equation*}\smallskip

\noindent hold for $k$ sufficiently large. Hence it follows from the inequalities (\ref{E:real-bounds}) that

\begin{align*}
x(n_k)&>2\log A-2\log(2mk^2+1)\pi-1\\
&>2\log A-2\log mk^2\pi-2\log \pi -1\\
&=-4\log k+(2\log A-2\log m\pi-2\log \pi-1)\\
&>-5\log k
\end{align*}
and
\begin{align*}
x(n_k)&<\log A-\log(2mk^2-1)\pi +1\\
&<\log A-\log mk^2\pi-\log 2+\log \pi+1\\
&=(\log A-\log m\pi+1)-2\log k-\log 2+\log \pi\\
&<-2\log k-3+\log \pi-\log 2\\
&<-2\log k-2,
\end{align*}\smallskip

\noindent completing the proof of the lemma.
\end{proof}

\begin{remark}\label{R:z-zero-rectangle}\rm We remark from the inequalities (\ref{E:real-imaginary-bounds-subseq-inequalities}) that the particular set of zeros $z(n_k)$ of equation (\ref{E:wright-equation}) must lie inside the rectangles whose vertices are given by the points $(-5\log k,\,-d_k)$, $(-2\log k-2,\,-d_k)$, $(-5\log k,\,-d_{k+1})$ and \newline $(-2\log k-2,\,-d_{k+1})$ in the complex $z$-plane.
\end{remark}

\section{\bf Zeros of an auxiliary function}

In the proof of our main results, it will become clear in \S 6 that we need to know the locations of zeros of the auxiliary function

\begin{equation}\label{E:g-definition}
g(\zeta)=\widehat{C}\e^{i\zeta}+\widehat{\sigma}\zeta^{\mu-\frac{1}{2}},
\end{equation}\smallskip

\noindent where $\mu,\,\widehat{C},\,\widehat{\sigma}$ are non-zero complex constants such that $\mu\neq \frac{1}{2}$, and $\zeta^{\mu-\frac{1}{2}}$ takes the principal branch.\footnote{The remaining case when $\mu=\frac{1}{2}$ will be discussed in the Lemma \ref{L:contour-special}.} We apply the results from \S \ref{S:location-zeros} to investigate the asymptotic locations of zeros for $g(\zeta)$. To do so, we first transform the equation $g(\zeta)=0$ into the form of (\ref{E:wright-equation}), where

\begin{equation}\label{E:z-a-definitions}
z=\frac{i\zeta}{\frac{1}{2}-\mu},\quad a=\frac{i}{\frac{1}{2}-\mu}(-\widehat{\sigma}\widehat{C}^{-1})^{\frac{1}{\frac{1}{2}-\mu}}.\footnote{We assume that $-\pi < \arg a \le \pi$.}
\end{equation}\smallskip

\noindent Let $\zeta(n)=u(n)+iv(n)$ and $\frac{1}{2}-\mu=b\e^{i\phi}$, where $n$ is an integer, $b=|\frac{1}{2}-\mu|>0$ and $-\pi<\phi \le \pi$. Then it follows from (\ref{E:z-a-definitions}) and $z(n)=x(n)+iy(n)$ that for sufficiently large positive or negative integers $n$,

\begin{align}\label{E:zeta-real-imaginary}
u(n) = (b\cos\phi) y(n)+(b\sin\phi) x(n), \quad v(n) = (b\sin\phi) y(n)-(b\cos\phi) x(n).
\end{align}\smallskip

In order to find precise asymptotic locations  of zeros of the function (\ref{E:g-definition}), we first consider the particular case that $\phi=\pi$ in (\ref{E:zeta-real-imaginary}). This {forces} $\mu>\frac{1}{2}$ and

\begin{equation*}
\zeta(n)=u(n)+iv(n)=b (-y(n)+ix(n)),
\end{equation*}\smallskip

\noindent where $b=\mu-\frac{1}{2}>0$. Therefore we obtain from Lemma \ref{L:real-imaginary-bounds-subseq} and Remark \ref{R:z-zero-rectangle} the following lemma.

\begin{lemma}\label{L:zeta-bounds} Let $\phi=\pi$ and $n_k$ be defined as in Lemma {\rm \ref{L:real-imaginary-bounds-subseq}}. Then for $k$ sufficiently large, we have

\begin{align}\label{E:zeta-bounds}
0<bd_k < u(n_k) < bd_{k+1},\quad -5b\log k< v(n_k)<-2b\log k-2b<0.
\end{align}\smallskip

\noindent In other words, the zeros $\zeta(n_k)$ must lie inside the rectangles $R_k$ whose vertices are given by the points $(bd_k,\,-5b\log k),\,(bd_k,\,-2b\log k-2b)$, $(bd_{k+1},\,-5b\log k)$ and $(bd_{k+1},\,-2b\log k-2b)$ in the $\zeta$-plane.
\end{lemma}\smallskip

\begin{remark}\label{R:fourth-quadrant}{\rm In view of the Lemma \ref{L:zeta-bounds}, we easily see that, when $\phi=\pi$, \textit{all} such zeros $\zeta(n_k)$ lie in the fourth quadrant of the $\zeta$-plane and the real part of each $\zeta(n_k)$ is increasing much faster than the imaginary part in such a way, so that the argument $\arg \zeta(n_k)$ is always negative and $\arg \zeta(n_k) \to 0$ as $n_k =-mk^2 \to -\infty$ (or as $k \to +\infty$).}
\end{remark}\smallskip

Now we are ready to define one of the contours that will be used in the proof of the Theorem \ref{T:Chiang-Yu}. For any given function $g$ of the form (\ref{E:g-definition}), the contour is formed by the curves $\Gamma_1(g),\,\Gamma_2(g)$ and the line segments $\ell_1(g),\,\ell_2(g)$ which are defined as follows:

\begin{equation}\label{E:case(i)-contour}
\begin{split}
\Gamma_1(g)&:=\big\{b(d_r-2i\log r) \,:\, k \le r \le 2k\big\},\\
\Gamma_2(g)&:=\big\{b(d_r-6i\log r) \,:\, k \le r \le 2k \big\},\\
\ell_1(g)&:=\big\{b(d_k+iv) \,:\, -6\log k \le v \le -2\log k \big\},\\
\ell_2(g)&:=\big\{b(d_{2k}+iv) \,:\, -6\log (2k) \le v \le -2\log (2k) \big\}.
\end{split}
\end{equation}\smallskip

\noindent We join the line segments $\ell_1(g),\,\ell_2(g)$ and the curves $\Gamma_1(g),\,\Gamma_2(g)$ to form the contour $\Omega(g,\,k)$ for each integer $k$. We then \textit{glue} the $\Omega(g,\,k)$ together along each pair of $\ell_1(g),\,\ell_2(g)$ and the resulting set is denoted by $\D \Omega(g)=\bigcup_{k=1}^{+\infty} \Omega(g,\,k)$. Then we define $\e^{i(\pi-\phi)}\Omega(g,\,k)$ and $ \e^{i(\pi-\phi)}\Omega(g)$ as follows:

\begin{equation}\label{E:definition-contour-set}
\begin{split}
\e^{i(\pi-\phi)}\Omega(g,\,k)&:=\big\{\e^{i(\pi-\phi)}\zeta \,:\, \zeta \in \Omega(g,\,k)\big\},\\ \e^{i(\pi-\phi)}\Omega(g)&:=\big\{\e^{i(\pi-\phi)}\zeta \,:\, \zeta \in \Omega(g) \big\}.
\end{split}
\end{equation}\smallskip

Thus we have the following result:

\begin{proposition}\label{L:g-number-zeros} Let $k$ be a large positive integer and $\mu \neq \frac{1}{2}$. Then the function $g(\zeta)$ as defined in {\rm (\ref{E:g-definition})} has at least $k$ distinct zeros lying inside the contour $\e^{i(\pi-\phi)}\Omega(g,\,k)$ and infinitely many zeros lie inside the set  $\e^{i(\pi-\phi)}\Omega(g)$.
\end{proposition}

\begin{proof} It suffices to prove the first statement. We suppose first that $\phi=\pi$ so that the contour and the sets given by (\ref{E:definition-contour-set}) are $\Omega(g,\,k)$ and $\Omega(g)$ respectively, and all the zeros $\zeta(n_k)$ lie in the fourth quadrant of the $\zeta$-plane by {Remark \ref{R:fourth-quadrant}}. We note that for each $j \in \{0,\,1,\ldots,\,k-1\}$, the vertices of the rectangle $R_{k+j}$ are given by $(bd_{k+j},\,-5b\log (k+j)),\,(bd_{k+j},\,-2b\log (k+j)-2b)$, $(bd_{k+j+1},\,-5b\log (k+j))$ and $(bd_{k+j+1},\,-2b\log (k+j)-2b)$. Since we have

\begin{equation*}
-2\log (k+j)-2<-2 \log r, \quad -6\log r<-5\log (k+j),
\end{equation*}\smallskip

\noindent where $k+j \le r \le k+j+1$, it means \textit{geometrically} that the upper (resp. lower) edge of $R_{k+j}$ is \textit{below} (resp. \textit{above}) the curve $\Gamma_1(g)$ (resp. $\Gamma_2(g)$), see Figure \ref{F:Figure1} for an illustration.\smallskip

\begin{figure*}[h]
\begin{center}
\includegraphics[width=\textwidth, height=0.35\textheight]{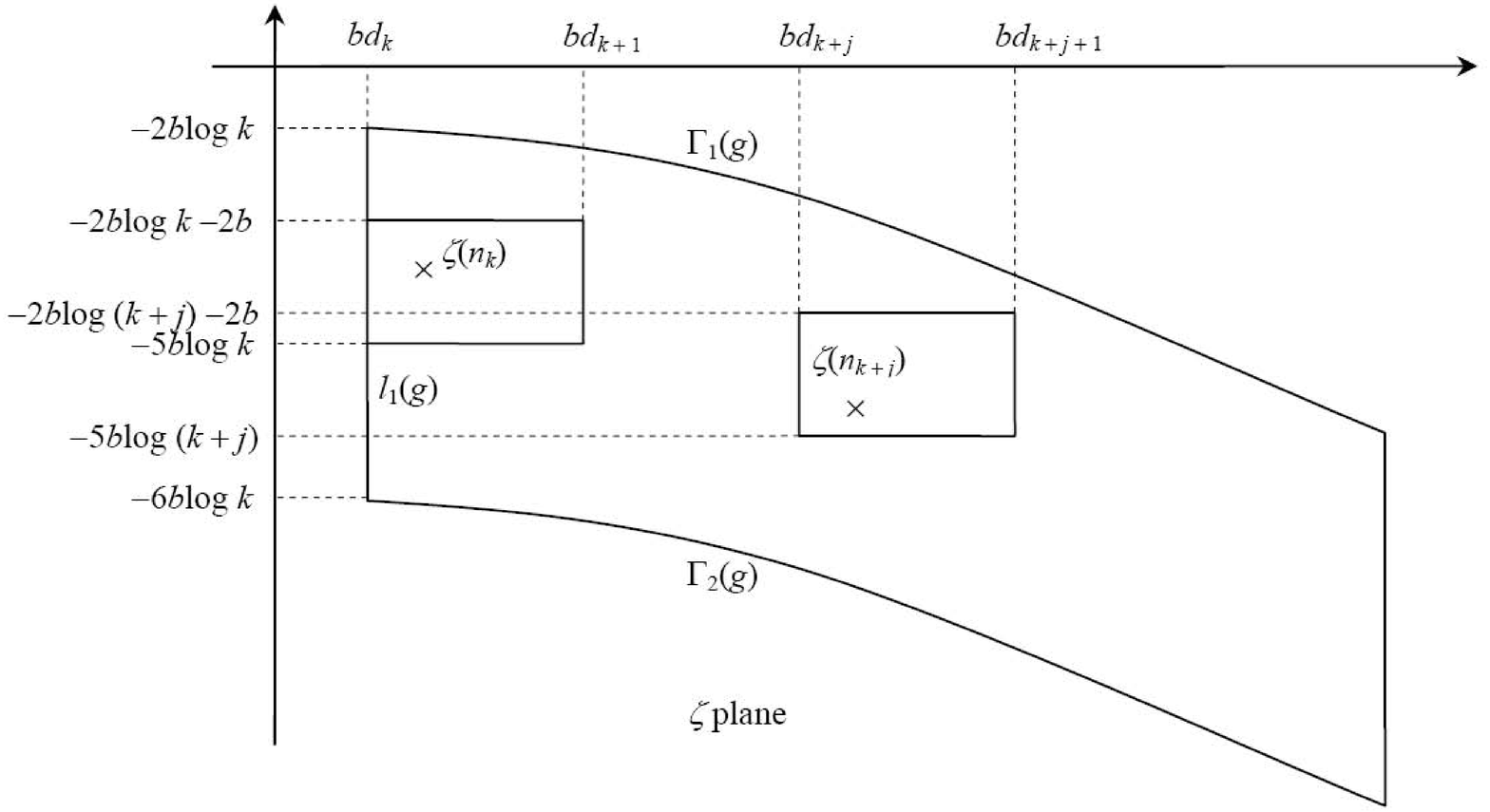}
\caption{The rectangles $R_{k+j}$ and a part of the contour $\Omega(g,\,k)$ when $\phi=\pi$.}\label{F:Figure1}
\end{center}
\end{figure*}

Thus the contour $\Omega(g,\,k)$ contains all the $k$ rectangles $R_{k+j}$, for $j \in \{0,\,1,\ldots,\,k-1\}$. By the Lemma \ref{L:zeta-bounds}, each rectangle $R_{k+j}$ contains the zero $\zeta(n_{k+j})$, for $j \in \{0,\,1,\ldots,\,k-1\}$. These zeros are \textit{distinct} because we have $R_{k+j}\cap R_{k+j'}=\emptyset$ whenever $j \neq j'$. Hence the result follows in this particular case.\smallskip

Next we suppose that $-\pi<\phi<\pi$. Then it may happen that \textit{not} all zeros $\zeta(n_k)$ lie in the fourth quadrant of the $\zeta$-plane. In this general case, we rotate the $\zeta$-plane through the angle $(\pi-\phi)$ to the $\zeta'$-plane, where $\zeta':=\e^{i(\pi-\phi)}\zeta$. Thus it follows from the relations (\ref{E:zeta-real-imaginary}) that

\begin{align*}
\zeta'(n_k)=\e^{i(\pi-\phi)}\zeta(n_k)=b(-y(n_k)+ix(n_k))
\end{align*}\smallskip

\noindent so that \textit{all} the zeros $\zeta'(n_k)$ {of $G(\zeta')=g(\e^{i(\phi-\pi)}\zeta')$} lie in the fourth quadrant of the $\zeta'$-plane by the Lemma \ref{L:zeta-bounds} and the Remark \ref{R:fourth-quadrant}. Thus the argument in the first part applies to this general case with respect to the contour $\e^{i(\pi-\phi)}\Omega(g,\,k)$ and the set $\e^{i(\pi-\phi)}\Omega(g)$, thus completing the proof of the proposition.
\end{proof}

\section{\bf Proof of Theorem \ref{T:Chiang-Yu}}

\subsection{Sufficiency part} Suppose that $f(z)$ is subnormal. Then the Remark \ref{R:Equivalent-conditions} asserts that we must have $A=B=0$ and one of the equations in (\ref{E:conditions}) holds. Thus according to (\ref{E:lommel-sol}), we have $\lambda(f)<+\infty$. This proves the sufficiency part of the theorem.\smallskip

\subsection{Necessary part} In order to complete the proof of the Theorem \ref{T:Chiang-Yu}, that is, to find the values of $\mu_j \pm \nu$ under the assumption that $\lambda(f)<+\infty$, we consider the function $f(z)$ in the form (\ref{E:f-solution-hankel}). Then we \textit{first} need the following result:

\begin{theorem}\label{T:exponent-of-convergence}  If $C \not=0$ or $D\not=0$, then $\lambda(f)=+\infty$.
\end{theorem}

\begin{proof}[Proof of Theorem \ref{T:exponent-of-convergence}]
We let $y(\zeta)$ in the form (\ref{E:general-soln-hankel}) be the general solution of (\ref{E:generalized-Lommel}). Since the Lemma \ref{L:lommel-independence} asserts that the Lommel functions $S_{\mu_j,\,\nu}(\zeta)$, $j=1,\,2,\ldots,\,n$, are linearly independent over $\mathbb{C}$ and that not all $\sigma_j$ are zero, so the summand in (\ref{E:general-soln-hankel}) is not identically zero.

\noindent Without loss of generality, we may assume that $\sigma_n \neq 0$ and the constants $\mu_j,\, j=1,\,2,\ldots,\,n$, in the Theorem \ref{T:Chiang-Yu} satisfy {$\Re(\mu_1)<\Re(\mu_2)<\cdots<\Re(\mu_n)$.} In order to prove the Theorem \ref{T:exponent-of-convergence}, we show that the general solution (\ref{E:general-soln-hankel}) has infinitely many zeros \textit{in the principal branch of $H_\nu^{(1)}(\zeta),\,H_\nu^{(2)}(\zeta)$ and $S_{\mu,\,\nu}(\zeta)$}.\footnote{That is $-\pi <\arg \zeta<\pi$.} The idea of our proof is to apply asymptotic expansions of the corresponding special functions and Rouch\'{e}'s theorem on suitably chosen contours.\smallskip

When $-\pi<\arg \zeta<\pi$, we substitute the asymptotic expansions (\ref{E:hankel1-asy}), (\ref{E:hankel2-asy}) and (\ref{E:lommel-asymptotic-expansion}) into the solution (\ref{E:general-soln-hankel}) to yield\footnote{This refers to the principal branch of the Hankel functions $H_{\nu}^{(1)}(\zeta),\, H_{\nu}^{(2)}(\zeta)$ and the Lommel functions $S_{\mu_j,\,\nu}(\zeta),\, j=1,\,2,\ldots,\,n$.}

\begin{align*}
\widehat{y}(\zeta) &=\bigg(\frac{\pi \zeta}{2}\bigg)^{\frac{1}{2}}y(\zeta) \notag\\
&=\widehat{C}\e^{i\zeta}\left[\sum_{k=0}^{p-1}\frac{(\nu,\,k)}{(-2i\zeta)^k}+O(\zeta^{-p}) \right]+\widehat{D}\e^{-i\zeta}\left[\sum_{k=0}^{p-1}\frac{(\nu,\,k)}{(2i\zeta)^k} +O(\zeta^{-p})\right] \notag\\
&\qquad+\sum_{j=1}^n \widehat{\sigma}_j \zeta^{\mu_j-\frac{1}{2}}\left[\sum_{k=0}^{p-1}\frac{(-1)^k c_{k,\,j}}{\zeta^{2k}} +O(\zeta^{-2p+1})\right],\notag
\end{align*}\smallskip

\noindent where $(\nu,\, k):=\frac{(-1)^k(\frac12-\nu)_k(\frac12+\nu)_k}{k!}$, $\widehat{C}:=C\e^{-i(\frac{1}{2}\nu\pi+\frac{1}{4}\pi)}$,
$\widehat{D}:=D\e^{i(\frac{1}{2}\nu\pi+\frac{1}{4}\pi)}$, $\D \widehat{\sigma}_j:=\sigma_j\Big(\frac{\pi}{2}\Big)^{\frac{1}{2}}$, where $j=1,\,2,\ldots,\,n$. This gives, when $p=1$, that

\begin{align}\label{E:Modify-soln}
\widehat{y}(\zeta)&=\widehat{C}\e^{i\zeta}\big[1+O(\zeta^{-1})\big]+\widehat{D}\e^{-i\zeta}\big[1+O(\zeta^{-1})\big]
+\sum_{j=1}^n\widehat{\sigma}_j\zeta^{\mu_j-\frac{1}{2}}\big[1+O(\zeta^{-1})\big].
\end{align}\smallskip

We distinguish two main cases: {\bf Case I}: $\mu_n \neq \frac{1}{2}$ and {\bf Case II}: $\mu_n=\frac{1}{2}$, which will then be further split into different subcases.\smallskip

\noindent {\bf Case I}: We suppose that $\mu_n \neq \frac{1}{2}$.\smallskip

Without loss of generality, we may assume that $C \neq 0$ so that $\widehat{C} \neq 0$. We choose the function (\ref{E:g-definition}) to be

\begin{equation}\label{E:g-reduced}
g(\zeta)=\widehat{C}\e^{i\zeta}+\widehat{\sigma}_n\zeta^{\mu_n-\frac{1}{2}},
\end{equation}\smallskip

\noindent where $\frac{1}{2}-\mu_n=b\e^{i\phi},\,b=|\frac{1}{2}-\mu_n|$ and $-\pi <\phi \le \pi$. Moreover, we assume that the chosen integer $m$ in the Lemma \ref{L:real-imaginary-bounds-subseq} also satisfies the inequality

\begin{equation}\label{E:m-condition}
{(bm\pi)^{|\Re(\mu_n)-\frac{1}{2}|}|\widehat{\sigma}_n|>2|\widehat{C}|\e^{|\Im(\mu_n)|\pi}.}
\end{equation}\smallskip

There are two subcases in {\bf Case I}. They are {\bf Subcase A}: $\phi=\pi$ and {\bf Subcase B}: $\phi \neq \pi$.\smallskip

\begin{itemize}
  \item[] {\bf Subcase A}: $\phi=\pi$. Then  the definition shows that $\frac{1}{2}-\mu_n=b\e^{i\pi}=-b$ which implies that $\mu_n$ must be a real number such that $\mu_n>\frac{1}{2}$ and $b=\mu_n-\frac{1}{2}>0$. We obtain from (\ref{E:Modify-soln}) and (\ref{E:g-reduced}) that

\begin{align}\label{E:Case-(i)-inequality}
|\widehat{y}(\zeta)-g(\zeta)|&=\bigg|\Big[\widehat{D}\e^{-i\zeta}+\widehat{C}\e^{i\zeta}O(\zeta^{-1})
+\widehat{D}\e^{-i\zeta}O(\zeta^{-1})\Big] \notag\\
&\qquad+O(\zeta^{\mu_n-\frac{3}{2}})+\sum_{j=1}^{n-1} \widehat{\sigma}_j \zeta^{\mu_j-\frac{1}{2}}\left[1+O(\zeta^{-1})\right]\bigg| \notag\\
&\le \Big|\widehat{D}\e^{-i\zeta}+\widehat{C}\e^{i\zeta}O(\zeta^{-1})+\widehat{D}\e^{-i\zeta}O(\zeta^{-1})\Big|
+O(|\zeta|^\kappa),
\end{align}\smallskip

\noindent where $\kappa$ is defined by

\begin{equation}\label{E:kappa-definition}
\kappa:=\max \bigg\{\Re(\mu_{n-1})-\frac{1}{2},\,\Re(\mu_n)-\frac{3}{2}\bigg\}<{\bigg|\mu_n-\frac{1}{2}\bigg|}=b.
\end{equation}\smallskip

We show that the inequality

\begin{equation}\label{E:main-inequality}
|\widehat{y}(\zeta)-g(\zeta)|<|g(\zeta)|
\end{equation}\smallskip

\noindent holds on the contour $\Omega(g,\,k)$ for all $k$ sufficiently large. In fact, it is always true that {$|\widehat{C}\e^{i\zeta}O(\zeta^{-1})|<|\widehat{C}\e^{i\zeta}|$, $|\widehat{D}\e^{-i\zeta}O(\zeta^{-1})|<|\widehat{D}\e^{-i\zeta}|$} and\smallskip

\begin{equation}\label{E:kappa-inequality}
|\zeta^\kappa| < |\zeta^{\frac{b+\kappa}{2}}|,
\end{equation}\smallskip

\noindent so it suffices to compare the values of {$|\widehat{D}\e^{-i\zeta}|,\, |\widehat{C}\e^{i\zeta}|$ and $|\zeta^{\frac{b+\kappa}{2}}|$} along the contour $\Omega(g,\,k)$. If $\zeta \in \ell_1(g)$, then we have $\zeta=b(d_k-i\gamma \log k)$ (see (\ref{E:case(i)-contour})), where $2 \le \gamma \le 6$; and if $\zeta \in \Gamma_1(g)$, then we have $\zeta=b(d_r-2i\log r)$, where $k \le r \le 2k$. We deduce that

\begin{align}\label{E:estimates-1}
|\e^{\pm i\zeta}|=\left\{
                    \begin{array}{ll}
                      k^{\pm b\gamma}, & \hbox{if $\zeta \in \ell_1(g)$;} \\
                      r^{\pm 2b}, & \hbox{if $\zeta \in \Gamma_1(g)$,}
                    \end{array}
                  \right.
\end{align}\smallskip

\noindent and for $k$ sufficiently large that

\begin{align}\label{E:estimates-2}
\left\{
  \begin{array}{ll}
    bm\pi k^2<|\zeta|<4bm\pi k^2, & \hbox{if $\zeta \in \ell_1(g)$;} \\
    bm\pi r^2<|\zeta|<4bm\pi r^2, & \hbox{if $\zeta \in \Gamma_1(g)$.}
  \end{array}
\right.
\end{align}\smallskip

On the one hand, for all sufficiently large $k$, it follows from (\ref{E:m-condition}) and (\ref{E:estimates-2}) that

\begin{align}\label{E:estimates-4}
\left\{
  \begin{array}{ll}
    2|\widehat{C}|k^{2b}<(bm\pi)^b |\widehat{\sigma}_n|  k^{2b} & \\
     \hspace{1.2cm}< |\widehat{\sigma}_n\zeta^b| < (4bm\pi)^b |\widehat{\sigma}_n|k^{2b}, & \hbox{if $\zeta \in \ell_1(g)$;} \\
    2|\widehat{C}|r^{2b}<(bm\pi)^b |\widehat{\sigma}_n| r^{2b} & \\
    \hspace{1.2cm}< |\widehat{\sigma}_n\zeta^b| < (4bm\pi)^b |\widehat{\sigma}_n| r^{2b}, & \hbox{if $\zeta \in \Gamma_1(g)$.}
  \end{array}
\right.
\end{align}\smallskip

\noindent But the triangle inequality

\begin{align}\label{E:triangle-inequality}
|g(\zeta)| \ge \big||\widehat{C}\e^{i\zeta}|-|\widehat{\sigma}_n\zeta^{\mu_n-\frac{1}{2}}|\big|,
\end{align}\smallskip

\noindent together with the relations (\ref{E:estimates-1}) and (\ref{E:estimates-4}) imply for $k\ge k_0$ for some positive integer $k_0$ that
\begin{align}\label{E:estimates-6}
|g(\zeta)|& \ge \left\{
                 \begin{array}{ll}
                   \D \big||\widehat{C}|k^{b\gamma}-|\widehat{\sigma}_n\zeta^b|\big|, & \hbox{if $\zeta \in \ell_1(g)$;} \\
                   \D \big||\widehat{\sigma}_n\zeta^b|-|\widehat{C}|r^{2b}\big|, & \hbox{if $\zeta \in \Gamma_1(g)$,}
                 \end{array}
               \right. \notag\\
          & > \left\{
                 \begin{array}{ll}
                   \D |\widehat{C}|k^{b\gamma}/2, & \hbox{if $\zeta \in \ell_1(g)$;} \\
                   \D |\widehat{C}|r^{2b}/2, & \hbox{if $\zeta \in \Gamma_1(g)$,}
                 \end{array}
               \right.
\end{align}\smallskip
where the lower estimate for the case $\zeta \in \ell_1(g)$ is trivial when $\gamma=2$ and the case when $\gamma>2$ follows since the factor $(4bm\pi)^b$ from (\ref{E:estimates-4}) is a constant. On the other hand, we obtain from {the relations (\ref{E:Case-(i)-inequality}), (\ref{E:kappa-inequality})}, (\ref{E:estimates-1}) and (\ref{E:estimates-2}) that

\begin{align}\label{E:estimates-3}
|\widehat{y}(\zeta)-g(\zeta)|&\le \left\{
   \begin{array}{ll}
     \D |\widehat{D}|k^{-b\gamma}+K_1\big(|\widehat{C}|k^{b\gamma}+|\widehat{D}|k^{-b\gamma}\big)k^{-2}+K_2|\zeta|^{\frac{b+\kappa}{2}}, & \hbox{if $\zeta \in \ell_1(g)$;}\\
     \D |\widehat{D}|r^{-2b}+K_3\big(|\widehat{C}|r^{2b}+|\widehat{D}|r^{-2b}\big)r^{-2}+K_4|\zeta|^{\frac{b+\kappa}{2}},& \hbox{if $\zeta \in \Gamma_1(g)$,}
   \end{array}
 \right. \notag\\
& \le \left\{
        \begin{array}{ll}
          \D |\widehat{D}|k^{-b\gamma}+K_5k^{b\gamma-2}+K_6k^{b+\kappa}, &\hbox{if $\zeta \in \ell_1(g)$;}  \\
          \D |\widehat{D}|r^{-2b}+K_7 r^{2b-2}+K_8r^{b+\kappa}, & \hbox{if $\zeta \in \Gamma_1(g)$,}
        \end{array}
      \right. \notag\\
& \le \left\{
        \begin{array}{ll}
          K_9 k^{\kappa_1}, & \hbox{if $\zeta \in \ell_1(g)$;} \\
          K_{10} r^{\kappa_2}, & \hbox{if $\zeta \in \Gamma_1(g)$,}
        \end{array}
      \right.
\end{align}\smallskip

\noindent where $\kappa_1:=\max\{-b\gamma,\,b\gamma-2,\, b+\kappa\}$, $\kappa_2:=\max\{-2b,\,2b-2,\, b+\kappa\}$ and $K_1,\, K_2,\ldots,\,K_{10}$ are some fixed positive constants depending only on $b,\, m$ (see Lemma \ref{L:real-imaginary-bounds-subseq}) and $k_0$. Note that it is easy to check that {$b\gamma>\kappa_1$ and $2b>\kappa_2$} hold trivially, we deduce from the inequalities (\ref{E:estimates-6}) and (\ref{E:estimates-3}) that the inequality (\ref{E:main-inequality}) holds on $\ell_1(g)$ and $\Gamma_1(g)$, and then similarly it also holds on $\ell_2(g)$ and $\Gamma_2(g)$. Hence the desired inequality (\ref{E:main-inequality}) holds on the contour $\Omega(g,\,k)$.\smallskip

  \item[] {\bf Subcase B}: If $\phi \neq \pi$, then it may happen as described in the proof of the Proposition \ref{L:g-number-zeros} that \textit{not} all zeros $\zeta(n_k)$ lie in the fourth quadrant of the $\zeta$-plane, where the integers $n_k$ are also defined in the Lemma \ref{L:real-imaginary-bounds-subseq}. However, one can rotate the $\zeta$-plane through the angle ($\pi-\phi$) as described in the Proposition \ref{L:g-number-zeros} (see also its proof) so that all such zeros can only lie in the fourth quadrant of the $\zeta'$-plane.\smallskip

      {In this circumstance}, we note that the inequalities (\ref{E:Case-(i)-inequality}) and (\ref{E:triangle-inequality}) are now replaced by\smallskip

      \begin{align}\label{E:Case-(i)-inequality-modify}
      |\widehat{Y}(\zeta')-G(\zeta')| &\le \Big| \widehat{D}\e^{-i\e^{i(\phi-\pi)}\zeta'}+\widehat{C}\e^{i\e^{i(\phi-\pi)}\zeta'}O(\zeta'^{-1}) \notag\\
      &\qquad+\widehat{D}\e^{-i\e^{i(\phi-\pi)}\zeta'}O(\zeta'^{-1})\Big|+O(|\zeta'|^\kappa)
      \end{align}\smallskip

      \noindent and\smallskip

      \begin{align}\label{E:triangle-inequality-modify}
      |G(\zeta')| \ge \Big| \big|\widehat{C}\e^{i\e^{i(\phi-\pi)}\zeta'}\big|-\big|\widehat{\sigma}_n\zeta'^{\Re(\mu_n)-\frac{1}{2}}\big|
      \e^{-\Im(\mu_n)\arg (\e^{i(\phi-\pi)}\zeta')} \Big|
      \end{align}\smallskip

      \noindent respectively, where $\widehat{Y}(\zeta')=\widehat{y}(\e^{i(\phi-\pi)}\zeta')$, $G(\zeta')=g(\e^{i(\phi-\pi)}\zeta')$ and the constant $\kappa$ is given by (\ref{E:kappa-definition}). Moreover, the relations (\ref{E:estimates-1}) and the inequalities (\ref{E:estimates-2}) are replaced by\smallskip

      \begin{align}\label{E:estimates-9}
      \big|\e^{\pm i\e^{i(\phi-\pi)}\zeta'}\big|&=\left\{
                                                    \begin{array}{ll}
                                                      k^{\pm b\gamma}, & \hbox{if $\zeta' \in \e^{i(\pi-\phi)}\ell_1(g)$;} \\
                                                      r^{\pm 2b}, & \hbox{if $\zeta' \in \e^{i(\pi-\phi)}\Gamma_1(g)$,}
                                                    \end{array}
                                                  \right.
      \end{align}\smallskip

      \noindent and\smallskip

      \begin{align}\label{E:estimates-5}
      \left\{
        \begin{array}{ll}
          bm\pi k^2<|\zeta'|<4bm\pi k^2, & \hbox{if $\zeta' \in \e^{i(\pi-\phi)}\ell_1(g)$;} \\
          bm\pi r^2<|\zeta'|<4bm\pi r^2, & \hbox{if $\zeta' \in \e^{i(\pi-\phi)}\Gamma_1(g)$}
        \end{array}
      \right.
      \end{align}\smallskip

      \noindent respectively, where $b=|\frac{1}{2}-\mu_n|>\Re(\mu_n)-\frac{1}{2}$, $2 \le \gamma \le 6$, $k \le r \le 2k$, where $\ell_1(g)$ and $\Gamma_1(g)$ are defined in (\ref{E:case(i)-contour}).\smallskip

      Now we further distinguish two cases between \textbf{(1)} $\Re(\mu_n)>\frac{1}{2}$ and \textbf{(2)} $\Re(\mu_n)<\frac{1}{2}$.\smallskip

      \textbf{(1)} If $\Re(\mu_n)>\frac{1}{2}$, then $\Re(\mu_n)-\frac{1}{2}>0$ and it follows from the inequalities (\ref{E:m-condition}) and (\ref{E:estimates-5}) that the inequalities (\ref{E:estimates-4}) are replaced by\smallskip

      \begin{align}\label{E:estimates-7}
      \left\{
      \begin{array}{ll}
      2|\widehat{C}|\e^{|\Im(\mu_n)|\pi}k^{2({\Re(\mu_n)-\frac{1}{2}})}< |\widehat{\sigma}_n\zeta'^{\Re(\mu_n)-\frac{1}{2}}| &\\
      \hspace{1.2cm}< (4bm\pi)^{\Re(\mu_n)-\frac{1}{2}} |\widehat{\sigma}_n|k^{2({\Re(\mu_n)-\frac{1}{2}})}, & \hbox{if $\zeta' \in \e^{i(\pi-\phi)}\ell_1(g)$;} \\
      &\\
      2|\widehat{C}|\e^{|\Im(\mu_n)|\pi}r^{2({\Re(\mu_n)-\frac{1}{2}})}< |\widehat{\sigma}_n\zeta'^{\Re(\mu_n)-\frac{1}{2}}| &\\
      \hspace{1.2cm}< (4bm\pi)^{\Re(\mu_n)-\frac{1}{2}} |\widehat{\sigma}_n| r^{2({\Re(\mu_n)-\frac{1}{2}})}, & \hbox{if $\zeta' \in \e^{i(\pi-\phi)}\Gamma_1(g)$} \\
      \end{array}
      \right.
      \end{align}\smallskip

      \noindent for $k$ sufficiently large. Thus the inequality (\ref{E:triangle-inequality-modify}) together with (\ref{E:estimates-7}) yield, for $k \ge k_1$ for some sufficiently large positive integer $k_1$, that\smallskip

      \begin{align}\label{E:estimates-8}
      |G(\zeta')|>\left\{
                    \begin{array}{ll}
                      |\widehat{C}|k^{b\gamma}/2, & \hbox{if $\zeta' \in \e^{i(\pi-\phi)}\ell_1(g)$;} \\
                      |\widehat{C}|r^{2b}/2, & \hbox{if $\zeta' \in \e^{i(\pi-\phi)}\Gamma_1(g)$.}
                    \end{array}
                  \right.
      \end{align}\smallskip

      \textbf{(2)} If $\Re(\mu_n)<\frac{1}{2}$, then we have $\frac{1}{2}-\Re(\mu_n)>0$ and the inequalities (\ref{E:estimates-4}) are now replaced by\smallskip

      \begin{align*}
      \left\{
      \begin{array}{ll}
      \D \frac{(4bm\pi)^{\Re(\mu_n)-\frac{1}{2}} |\widehat{\sigma}_n|}{k^{2(\frac{1}{2}-\Re(\mu_n))}}< |\widehat{\sigma}_n\zeta'^{\Re(\mu_n)-\frac{1}{2}}|<\frac{(bm\pi)^{\Re(\mu_n)-\frac{1}{2}} |\widehat{\sigma}_n|}{k^{2(\frac{1}{2}-\Re(\mu_n))}}, & \hbox{if $\zeta' \in \e^{i(\pi-\phi)}\ell_1(g)$;} \\
      \D \frac{(4bm\pi)^{\Re(\mu_n)-\frac{1}{2}} |\widehat{\sigma}_n|}{r^{2(\frac{1}{2}-\Re(\mu_n))}}< |\widehat{\sigma}_n\zeta'^{\Re(\mu_n)-\frac{1}{2}}|<\frac{(bm\pi)^{\Re(\mu_n)-\frac{1}{2}} |\widehat{\sigma}_n|}{r^{2(\frac{1}{2}-\Re(\mu_n))}}, & \hbox{if $\zeta' \in \e^{i(\pi-\phi)}\Gamma_1(g)$}
      \end{array}
      \right.
      \end{align*}\smallskip

      \noindent for $k$ sufficiently large. Therefore we deduce from these that\smallskip

      \begin{align*}
      \lim_{k \to +\infty \atop{\zeta' \in \e^{i(\pi-\phi)}\ell_1(g)}}|\widehat{\sigma}_n\zeta'^{\Re(\mu_n)-\frac{1}{2}}|=0 \quad \mbox{and}\quad
      \lim_{r \to +\infty \atop{\zeta' \in \e^{i(\pi-\phi)}\Gamma_1(g)}}|\widehat{\sigma}_n\zeta'^{\Re(\mu_n)-\frac{1}{2}}|=0.
      \end{align*}\smallskip

      On the one hand, these limits show that the inequality (\ref{E:triangle-inequality-modify}) imply for $k \ge k_2$ for some sufficiently large positive integer $k_2$ that the inequalities (\ref{E:estimates-8}) hold in this case. On the other hand, it follows from the relations (\ref{E:kappa-inequality}), (\ref{E:Case-(i)-inequality-modify}), (\ref{E:estimates-9}) and (\ref{E:estimates-5}) that the inequalities\smallskip

      \begin{align}\label{E:estimates-10}
      |\widehat{Y}(\zeta')-G(\zeta')|\le\left\{
                                          \begin{array}{ll}
                                            K_{11}k^{\kappa_1}, & \hbox{if $\zeta' \in \e^{i(\pi-\phi)}\ell_1(g)$;} \\
                                            K_{12}r^{\kappa_2}, & \hbox{if $\zeta' \in \e^{i(\pi-\phi)}\Gamma_1(g)$}
                                          \end{array}
                                        \right.
      \end{align}\smallskip

      \noindent hold in this \textbf{Subcase B}, where $K_{11},\, K_{12}$ are some positive constants, $\kappa_1=\max\{-b\gamma,\,b\gamma-2,\,b+\kappa\}$ and $\kappa_2=\max\{-2b,\,2b-2,\,b+\kappa\}$ so that (\ref{E:estimates-8}) and (\ref{E:estimates-10}) imply the inequality

      \begin{align*}
      |\widehat{Y}(\zeta')-G(\zeta')|<|G(\zeta')|
      \end{align*}\smallskip

      \noindent holds on the contour $\e^{i(\pi-\phi)}\Omega(g,\,k)$ for all $k$ sufficiently large. Hence, our desired inequality (\ref{E:main-inequality}) still holds in this general case after we transform the $\zeta'$-plane back to the $\zeta$-plane.
\end{itemize} \smallskip

\noindent {\bf Case II}: Next, we suppose that $\mu_n=\frac{1}{2}$.\smallskip

Unfortunately, the contour $\Omega(g,\,k)$ and the auxiliary function defined in (\ref{E:case(i)-contour}) and  (\ref{E:g-reduced}) respectively, do not seem to apply in this case. This is because the zeros of (\ref{E:g-reduced}) distribute evenly on a straight line parallel to the real axis, and hence the region enclosed by the contour $\Omega(g,\,k)$ can only contain \textit{finitely many} such zeros for every positive integer $k$. We choose the alternative auxiliary function to be:
	\begin{equation}\label{E:g-hat-definition}
		\widehat{g}(\zeta)=\widehat{C}\e^{i\zeta}+\widehat{D}\e^{-i\zeta}+\widehat{\sigma}_n.
	\end{equation}\smallskip
Without loss of generality, we continue to assume that $\widehat{C} \neq 0$. It remains to construct a suitable contour that contains the zeros of (\ref{E:g-hat-definition}) which are given by the following lemma.

      \begin{lemma}\label{L:contour-special} Let $k$ be an integer.
      \begin{enumerate}
      \item[(a)] If $\widehat{D} \neq 0$, then the zeros of {\rm (\ref{E:g-hat-definition})} are given by
      \begin{equation}\label{E:zeros-g-hat-nonzero}
      \zeta_\pm(k) =2k\pi+\theta_\pm+i\log |\Delta_\pm|,
      \end{equation}\smallskip
      \noindent where $\Delta_\pm^{-1}$ are solutions of the quadratic equation $\widehat{C}x^2+\widehat{\sigma}_nx+\widehat{D}=0$, which are given by $\Delta_\pm^{-1}:=\Big(-\widehat{\sigma}_n \pm \sqrt{\widehat{\sigma}_n^2-4\widehat{C}\cdot\widehat{D}}\Big)/2\widehat{C}$ and $\theta_\pm$ are the principal arguments of $\Delta_\pm^{-1}$.
        \item[(b)] If $\widehat{D}=0$, then the zeros of {\rm (\ref{E:g-hat-definition})} are given by
      \begin{equation}\label{E:zeros-g-hat-zero}
      \zeta_0(k)=2k\pi+\theta_0+i\log|\Delta_0|,
      \end{equation}\smallskip
      \noindent where $\Delta_0^{-1}:=-\widehat{\sigma}_n/\widehat{C}$ and $\theta_0$ is the principal argument of $\Delta_0^{-1}$.
      \end{enumerate}
      \end{lemma}
We omit its proof.
      \begin{remark}\label{R:contour-special-zeros} \rm We remark that none of the $\Delta_+$, $\Delta_-$ or $\Delta_0$ can be zero. Otherwise, $\widehat{D}$ or $\widehat{C}$ would be zero which contradicts the assumption. Moreover, let $L_+$ and $L_-$ be two horizontal straight lines on which the zeros of the equation (\ref{E:zeros-g-hat-nonzero}) fall on, such that $L_+$ corresponds to the zeros of $\zeta_+$ and  $L_-$ corresponds to the zeros of $\zeta_-$ in Lemma \ref{L:contour-special}(a). Similarly, we let $L_0$ denote the straight line representing the zeros of the equation (\ref{E:zeros-g-hat-zero}) in Lemma \ref{L:contour-special}(b). Both $L_0,\,L_+$ and $L_-$ are parallel to the real axis in the $\zeta$-plane.
      \end{remark} \smallskip

      The construction of the contour is divided into different cases depending on whether the $D$ vanishes. They are, {\bf Subcase A}: $\widehat{D} \neq 0$ and {\bf Subcase B}: $\widehat{D}=0$. The {\bf Subcase A} is further divided into {\bf (1)} $|\Delta_+| \neq |\Delta_-|$ and {\bf (2)} $|\Delta_+| = |\Delta_-|$. The {\bf (2)} is divided into  (i) $\theta_+=\theta_-$ and (ii) $\theta_+ \neq \theta_-$.
\smallskip

      Now we can start the construction of the contour.\smallskip

      \begin{itemize}
        \item[] {\bf Subcase A}: Suppose that $\widehat{D} \neq 0$. By the Remark \ref{R:contour-special-zeros}, we define a constant $d$ as follows:

            \begin{equation}\label{E:d-definition}
            d:=\left\{
            \begin{array}{ll}
              \big|\log |\Delta_+|-\log |\Delta_-|\big|, & \hbox{if $|\Delta_+| \neq |\Delta_-|$;} \\
             1, & \hbox{if $|\Delta_+| = |\Delta_-|$.}
            \end{array}
            \right.
            \end{equation}\smallskip

            \noindent It is easy to see from (\ref{E:d-definition}) that we must have $d>0$. We distinguish two cases between \textbf{(1)} $|\Delta_+| \neq |\Delta_-|$ and \textbf{(2)} $|\Delta_+|=|\Delta_-|$. \smallskip

            \noindent \textbf{(1)} If $|\Delta_+| \neq |\Delta_-|$, then we define, for each integer $k\ge k_0$ for some suitably large positive integer $k_0$, the line segments $\ell_1(\widehat{g}),\,\ell_2(\widehat{g}),\,\ell_3(\widehat{g})$ and $\ell_4(\widehat{g})$ as follows:

            \begin{equation}\label{E:case(ii)-contour}
            \begin{split}
            \ell_1(\widehat{g})&:=\bigg\{(2k-1)\pi+\theta_+ +i\bigg(\log|\Delta_+|+\frac{dy}{2}\bigg) \,:\, -1 \le y \le 1\bigg\},\\
            \ell_2(\widehat{g})&:=\bigg\{(4k+1)\pi+\theta_+ +i\bigg(\log|\Delta_+|+\frac{dy}{2}\bigg) \,:\, -1 \le y \le 1\bigg\},\\
            \ell_3(\widehat{g})&:=\bigg\{x+\theta_+ +i\bigg(\log|\Delta_+|-\frac{d}{2}\bigg) \,:\, (2k-1)\pi \le x \le (4k+1)\pi\bigg\},\\
            \ell_4(\widehat{g})&:=\bigg\{x+\theta_+ +i\bigg(\log|\Delta_+|+\frac{d}{2}\bigg) \,:\, (2k-1)\pi\le x \le (4k+1)\pi\bigg\}.
            \end{split}
            \end{equation}\smallskip

            \noindent Then the line segments $\ell_j(\widehat{g})\,(j=1,\,2,\,3,\,4)$ are concatenated to form the rectangular contour $\Omega(\widehat{g},\,k)$.  We also form the set $\D \Omega(\widehat{g})=\bigcup_{k=1}^{+\infty} \Omega(\widehat{g},\,k)$, see Figure 2.\smallskip

      \begin{figure*}[h]
      \begin{center}
      \includegraphics[width=0.6\textwidth, height=0.36\textheight]{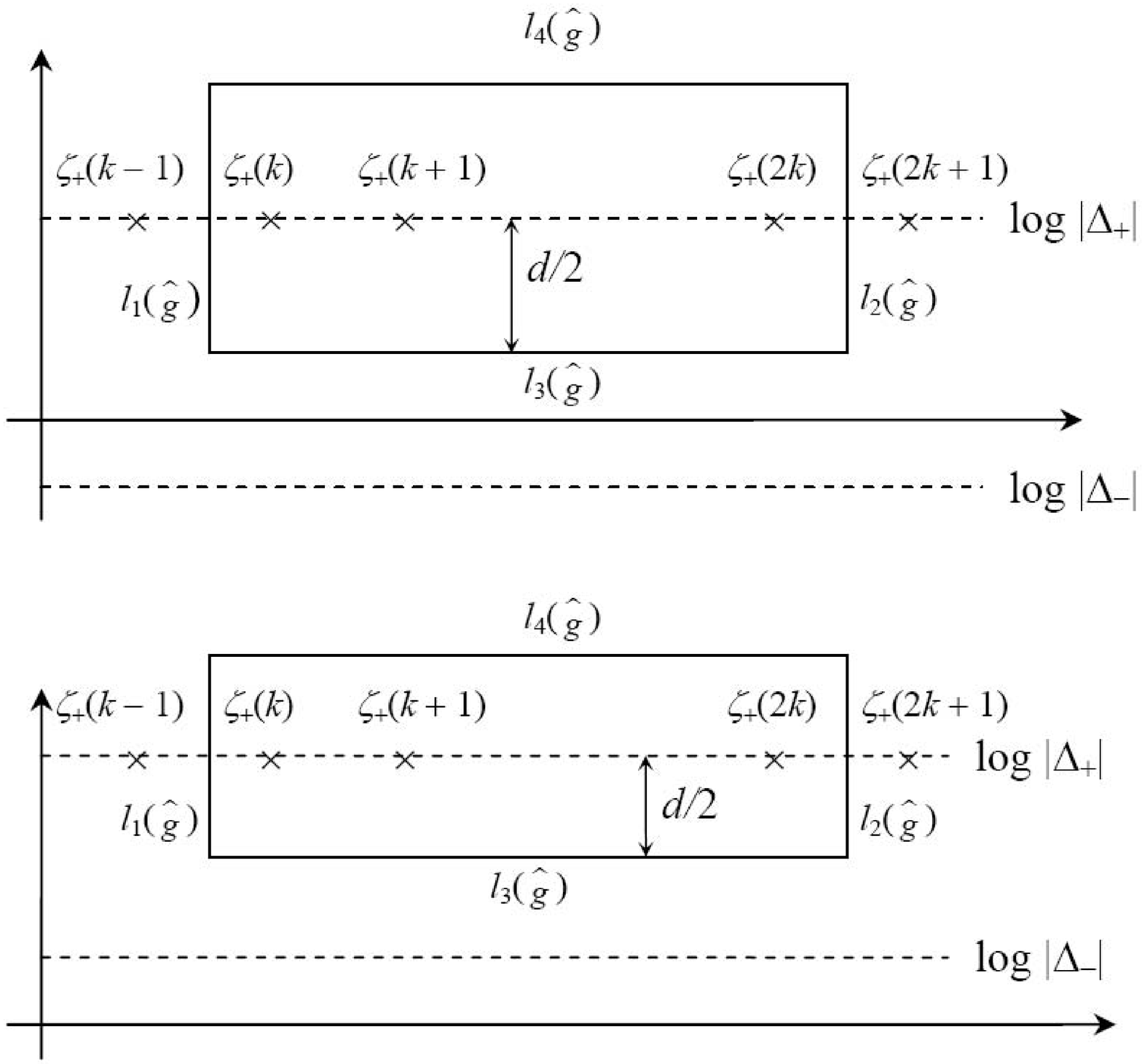}
      \caption{The contour $\Omega(\widehat{g},\,k)$ when $|\Delta_+| \neq |\Delta_-|$.}\label{F:Figure2}
      \end{center}
      \end{figure*}\smallskip

       Instead of inequality (\ref{E:main-inequality}), we shall show that the inequality

      \begin{equation}\label{E:main-inequality-hat}
      |\widehat{y}(\zeta)-\widehat{g}(\zeta)|<|\widehat{g}(\zeta)|
      \end{equation}\smallskip

      \noindent holds on $\Omega(\widehat{g},\,k)$ for all $k$ sufficiently large.\smallskip

      On the one hand, we note from the Lemma \ref{L:contour-special}(a), the Remark \ref{R:contour-special-zeros} and the definition (\ref{E:case(ii)-contour}) that for every integer $k$, exactly $k+1$ distinct zeros of $\widehat{g}(\zeta)$ lie \textit{inside} the $\Omega(\widehat{g},\,k)$ but all $\zeta_-$ lie \textit{outside} the $\Omega(\widehat{g},\,k)$, see Figure \ref{F:Figure2} for an illustration. Therefore we must have the fact that  $\widehat{g}(\zeta)$ \textit{does not} pass through any zero along the $\Omega(\widehat{g},\,k)$ for every positive integer $k$. In other words, there exists a positive constant $\Delta_1(k)$, depending only on $k$, such that the inequality

      \begin{align}\label{E:relation-2}
      |\widehat{g}(\zeta)|>\Delta_1(k)>0
      \end{align}\smallskip

      \noindent holds on $\Omega(\widehat{g},\,k)$ for every positive integer $k$. To obtain the desired inequality (\ref{E:main-inequality-hat}), we must show that the constant $\Delta_1(k)$ can be chosen \textit{independent} of $k$. To see this, we note that for each positive integer $k$, we have

      \begin{align}\label{E:relation-3}
      \e^{\pm i\zeta}=\left\{
                        \begin{array}{ll}
                           -|\Delta_+|^{\mp 1}\e^{\mp \frac{dy}{2}}\e^{\pm i\theta_+}, & \hbox{if $\zeta \in \ell_1(\widehat{g}) \cup \ell_2(\widehat{g})$;} \\
                           |\Delta_+|^{\mp 1}\e^{\pm \frac{d}{2}}\e^{\pm i(x+\theta_+)}, & \hbox{if $\zeta \in \ell_3(\widehat{g})$;} \\
                           |\Delta_+|^{\mp 1}\e^{\mp \frac{d}{2}}\e^{\pm i(x+\theta_+)}, & \hbox{if $\zeta \in \ell_4(\widehat{g})$,}
                        \end{array}
                      \right.
      \end{align}\smallskip

      \noindent and so

      \begin{align}\label{E:relation-4}
      \widehat{g}(\zeta)=\left\{
                           \begin{array}{ll}
                             -\widehat{C}|\Delta_+|^{-1}\e^{-\frac{dy}{2}}\e^{i\theta_+}
                             -\widehat{D}|\Delta_+|\e^{\frac{dy}{2}}\e^{-i\theta_+}+\widehat{\sigma}_n, & \hbox{if $\zeta \in \ell_1(\widehat{g}) \cup \ell_2(\widehat{g})$;} \\
                             \widehat{C}|\Delta_+|^{-1}\e^{\frac{d}{2}}\e^{i(x+\theta_+)}
                             & \\
                             \qquad  +\widehat{D}|\Delta_+|\e^{-\frac{d}{2}}\e^{-i(x+\theta_+)} +\widehat{\sigma}_n, & \hbox{if $\zeta \in \ell_3(\widehat{g})$;} \\
                             \widehat{C}|\Delta_+|^{-1}\e^{-\frac{d}{2}}\e^{i(x+\theta_+)}
                             & \\
                             \qquad +\widehat{D}|\Delta_+|\e^{\frac{d}{2}}\e^{-i(x+\theta_+)}
                             +\widehat{\sigma}_n, & \hbox{if $\zeta \in \ell_4(\widehat{g})$.}
                           \end{array}
                         \right.
      \end{align}\smallskip

      \noindent Hence this implies that $\Delta_1(k)$ can be chosen \textit{independent} of $k$. We denote this positive number to be $\Delta_1$, thus we have the inequality

      \begin{equation}\label{E:relation-5}
      |\widehat{g}(\zeta)|>\Delta_1>0
      \end{equation}\smallskip

      \noindent holds on $\Omega(\widehat{g},\,k)$ for every positive integer $k$.\smallskip

      On the other hand, it follows from (\ref{E:Modify-soln}) and the chosen function (\ref{E:g-hat-definition}) that

      \begin{align}\label{E:relation-6}
      |\widehat{y}(\zeta)-\widehat{g}(\zeta)|&=\Big|\widehat{C}\e^{i\zeta}O(\zeta^{-1})
      +\widehat{D}\e^{-i\zeta}O(\zeta^{-1})+O(\zeta^{-1})+
      \sum_{j=1}^{n-1}\widehat{\sigma}_j\zeta^{\mu_j-\frac{1}{2}}[1+O(\zeta^{-1})]\Big| \notag \\
      &\le \frac{\Delta_2|\e^{i\zeta}|}{|\zeta|}+\frac{\Delta_3|\e^{-i\zeta}|}{|\zeta|}+O(|\zeta|^{-1}) +\sum_{j=1}^{n-1} O(|\zeta|^{\Re(\mu_j)-\frac{1}{2}}),
      \end{align}\smallskip

      \noindent where $\Delta_2$ and $\Delta_3$ are some fixed positive constants. Since $\mu_n=\frac{1}{2}$, the definition (\ref{E:kappa-definition}) implies that $\kappa$ is negative and $\kappa \ge -1$, $\kappa \ge \Re(\mu_j)-\frac{1}{2}$ for all $1 \le j \le n-1$. Thus the relations (\ref{E:relation-3}) and (\ref{E:relation-6}) imply that

      \begin{align}\label{E:relation-7}
      |\widehat{y}(\zeta)-\widehat{g}(\zeta)|&\le \frac{\Delta_2|\e^{i\zeta}|}{|\zeta|}+\frac{\Delta_3|\e^{-i\zeta}|}{|\zeta|}+O(|\zeta|^\kappa) \notag\\
      & < \frac{\Delta_4}{k\pi}+\frac{\Delta_5}{(k\pi)^{|\kappa|}},
      \end{align}\smallskip

      \noindent holds on the contour $\Omega(\widehat{g},\,k)$, where $\Delta_4$ and $\Delta_5$ are two fixed positive constants independent of $k$. Hence we obtain from (\ref{E:relation-5}) and (\ref{E:relation-7}) that the desired inequality (\ref{E:main-inequality-hat}) holds on $\Omega(\widehat{g},\,k)$ for all sufficiently large $k$.\medskip

      \noindent \textbf{(2)} If $|\Delta_+|=|\Delta_-|$, then the definition (\ref{E:d-definition}) gives $d=1$. In addition, all the $\zeta_+(k)$ and $\zeta_-(k)$ lie on the same straight line $L=L_+=L_-$ (see Remark \ref{R:contour-special-zeros}) so that for every integer $k$, we have $|\zeta_+(k)-\zeta_-(k)|=|\theta_+-\theta-|$. Thus, there are two possibilities: (i) $\theta_+=\theta_-$ and (ii) $\theta_+ \neq \theta_-$.\smallskip

      \begin{itemize}
        \item[(i)] If $\theta_+=\theta_-$, then $\zeta_+(k)=\zeta_-(k)$ for every integer $k$. Hence the above contour $\Omega(\widehat{g},\,k)$ and the argument leading to the inequality (\ref{E:main-inequality-hat}) can be applied without any change.\smallskip

        \item[(ii)] If $\theta_+ \neq \theta_-$, then it may happen that $\ell_1(\widehat{g})$ or $\ell_2(\widehat{g})$ passes through the zeros $\zeta_-(k-1),\,\zeta_-(k),\,\zeta_-(2k)$ or $\zeta_-(2k+1)$, so we need to modify the contour $\Omega(\widehat{g},\,k)$ defined in (\ref{E:case(ii)-contour}). In fact, we can replace $(2k-1)\pi$ and $(4k+1)\pi$ by $2k\pi-(\theta_+-\theta_-)/2$ and $4k\pi-(\theta_+-\theta_-)/2$ respectively in the definitions (\ref{E:case(ii)-contour}). We then denote the modified line segments by $\ell_1'(\widehat{g}),\,\ell_2'(\widehat{g}),\,\ell_3'(\widehat{g})$ and $\ell_4'(\widehat{g})$ respectively:

            \begin{equation*}\label{E:case(ii)-modified-contour}
            \begin{split}
            \ell_1'(\widehat{g})&:=\bigg\{2k\pi-\frac{\theta_+-\theta_-}{2}+\theta_+ +i\bigg(\log|\Delta_+|+\frac{dy}{2}\bigg) \,:\, -1 \le y \le 1\bigg\},\\
            \ell_2'(\widehat{g})&:=\bigg\{4k\pi-\frac{\theta_+-\theta_-}{2}+\theta_+ +i\bigg(\log|\Delta_+|+\frac{dy}{2}\bigg) \,:\, -1 \le y \le 1\bigg\},\\
            \ell_3'(\widehat{g})&:=\bigg\{x+\theta_+ +i\bigg(\log|\Delta_+|-\frac{d}{2}\bigg) \,:\, 2k\pi-\frac{\theta_+-\theta_-}{2} \le x \le 4k\pi-\frac{\theta_+-\theta_-}{2}\bigg\},\\
            \ell_4'(\widehat{g})&:=\bigg\{x+\theta_+ +i\bigg(\log|\Delta_+|+\frac{d}{2}\bigg) \,:\, 2k\pi-\frac{\theta_+-\theta_-}{2} \le x \le 4k\pi-\frac{\theta_+-\theta_-}{2}\bigg\}.
            \end{split}
            \end{equation*}\smallskip

            Then the contour and the infinite strip are defined similarly and denoted by $\Omega'(\widehat{g},\,k)$ and $\Omega'(\widehat{g})$ respectively. See Figure 3 below.

            \begin{figure*}[h]
            \begin{center}
            \includegraphics[width=0.85\textwidth, height=0.2\textheight]{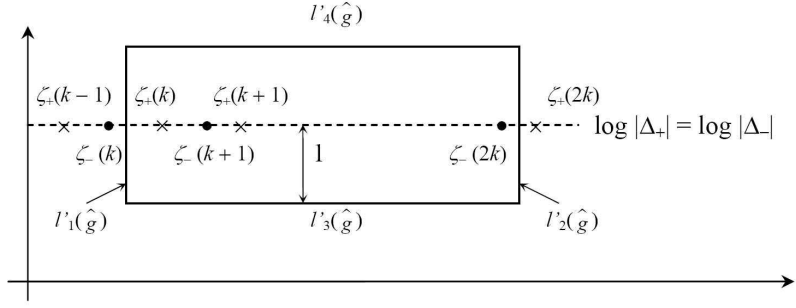}
            \caption{The modified contour $\Omega'(\widehat{g},\,k)$ when $|\Delta_+|=|\Delta_-|$.}\label{F:Figure3}
            \end{center}
            \end{figure*}\smallskip

            \noindent Since  $d=1$ in (\ref{E:d-definition}), thus the relations (\ref{E:relation-3}) and (\ref{E:relation-4}) are replaced by

            \begin{align*}
            \e^{\pm i\zeta}=\left\{
                        \begin{array}{ll}
                          |\Delta_+|^{\mp 1} \e^{\mp \frac{y}{2}} \e^{\pm (\frac{\theta_+ +\theta-}{2})i}, & \hbox{if $\zeta \in \ell_1'(\widehat{g}) \cup \ell_2'(\widehat{g})$}; \\
                          |\Delta_+|^{\mp 1}\e^{\pm \frac{1}{2}} \e^{\pm i(x+\theta_+)}, & \hbox{if $\zeta \in \ell_3'(\widehat{g})$;} \\
                          |\Delta_+|^{\mp 1}\e^{\mp \frac{1}{2}} \e^{\pm i(x+\theta_+)}, & \hbox{if $\zeta \in \ell_4'(\widehat{g})$,}
                        \end{array}
                      \right.
      \end{align*}\smallskip

      \noindent and

      \begin{align*}
      \widehat{g}(\zeta)=\left\{
                           \begin{array}{ll}
                             \widehat{C}|\Delta_+|^{-1} \e^{-\frac{y}{2}}\e^{i(\theta_+ +\theta-)/2} \\
                             \qquad+\widehat{D}|\Delta_+| \e^{\frac{y}{2}} \e^{-i(\theta_+ +\theta-)/2}+\widehat{\sigma}_n, & \hbox{if $\zeta \in \ell_1'(\widehat{g}) \cup \ell_2'(\widehat{g})$;} \\
                             \widehat{C}|\Delta_+|^{-1}\e^{\frac{1}{2}}\e^{i(x+\theta_+)}
                             & \\
                             \qquad  +\widehat{D}|\Delta_+|\e^{-\frac{1}{2}}\e^{-i(x+\theta_+)} +\widehat{\sigma}_n, & \hbox{if $\zeta \in \ell_3'(\widehat{g})$;} \\
                             \widehat{C}|\Delta_+|^{-1}\e^{-\frac{1}{2}}\e^{i(x+\theta_+)}
                             & \\
                             \qquad +\widehat{D}|\Delta_+|\e^{\frac{1}{2}}\e^{-i(x+\theta_+)}
                             +\widehat{\sigma}_n, & \hbox{if $\zeta \in \ell_4'(\widehat{g})$}
                           \end{array}
                         \right.
      \end{align*}\smallskip

            \noindent respectively. Thus, the inequalities (\ref{E:relation-5}), (\ref{E:relation-6}) and (\ref{E:relation-7}) can be similarly deduced with a possibly different set of the positive constants $\Delta_1,\,\Delta_2,\ldots,\,\Delta_5$.
      \end{itemize}

      \begin{remark}\rm It is trivial to check that there are totally $2k+1$ distinct zeros inside the modified contour $\Omega'(\widehat{g},\,k)$ for every positive integer $k$.
      \end{remark}

        \item[] {\bf Subcase B}: Suppose that $\widehat{D}=0$. Then it is easy to see that this can be regarded as the degenerated case in \textbf{Subcase A(1)(i)} with the constant $|\Delta_+|$ and the straight line $L_+$ replaced by $|\Delta_0|$ and $L_0$ respectively. \smallskip
      \end{itemize}\smallskip

We may now continue the proof of the Theorem \ref{T:exponent-of-convergence}.\newline

So Rouch\'e's theorem implies that, the functions $\widehat{y}(\zeta)$ and $g(\zeta)$ (resp. $\widehat{g}(\zeta)$) have the same number of zeros inside $\Omega(g,\,k)$ (resp. $\Omega(\widehat{g},\,k)$ or $\Omega'(\widehat{g},\,k)$). The Proposition \ref{L:g-number-zeros} (resp. Lemma \ref{L:contour-special}) asserts that $g(\zeta)$ (resp. $\widehat{g}(\zeta)$), and hence $\widehat{y}(\zeta)$, has \textit{infinitely many distinct zeros} inside $\Omega(g)$ (resp. $\Omega(\widehat{g})$ or $\Omega'(\widehat{g})$). Let $\mathfrak{n}(D,\, f)$ denote the number of zeros of the function $f(z)$ inside the set $D$. Then given any $0<\varepsilon<1$, there exists an infinite sequence $\{\zeta_n\}$ of zeros of $\widehat{y}(\zeta)$, and hence of $y(\zeta)$, with $|\zeta_n|=\rho_n$ inside $\Omega(g)$ (resp. $\Omega(\widehat{g})$ or $\Omega'(\widehat{g})$) such that\smallskip

\begin{align*}
\mathfrak{n}\left(\Omega(g),\, y(\zeta)\right) \ge \rho_n^{1-\varepsilon}& \quad (\mbox{resp. $\mathfrak{n}\left(\Omega(\widehat{g}),\, y(\zeta)\right) \ge \rho_n^{1-\varepsilon}$ or $\mathfrak{n}\left(\Omega'(\widehat{g}),\, y(\zeta)\right) \ge \rho_n^{1-\varepsilon}$})
\end{align*}\smallskip

\noindent for all sufficiently large $n$. By the substitution $L\e^{Mz}=\zeta$, where $z=r\e^{i\theta}$ and $\zeta=\rho \e^{i\varphi}$. Then for choosing $r_n$ {and $\theta_n$ such that $r_n \to +\infty$ as $n \to +\infty$ and $\theta_n+b=0$ for all positive integers $n$, where $b$ is the principal argument of $M$, we must have $\rho_n=|L|\e^{|M|r_n} \to +\infty$ as $n \to +\infty$ and then}\smallskip

\begin{align*}
\frac{\log \mathfrak{n} \Big( \left\{z \,:\, |z| \le \frac{1}{|M|}\log\rho_n\right\},\,f(z)\Big)}{\log r_n}&\ge\frac{\log\mathfrak{n}\left( \left\{\zeta \,:\,\frac{|L|}{\rho_n} \le |\zeta| \le |L|\rho_n,\, \arg \zeta \not =\pi\right\},\, y(\zeta) \right)}{\log r_n}\\
&\ge \frac{\log (\rho_n)^{1-\epsilon}}{\log r_n}\\
&=\frac{(1-\epsilon)\log \rho_n}{\log \log \rho_n} \to +\infty
\end{align*}\smallskip

\noindent as $n \to +\infty$ which implies that $\lambda (f)=+\infty$, thus completing the proof of the Theorem.
\end{proof}

We can continue the proof of the Theorem \ref{T:Chiang-Yu} now.\newline

Recall that $f(z)=\e^{-Nz}y(Le^{Mz})$, where $y(\zeta)$ is a solution to the equation (\ref{E:generalized-Lommel}). So the requirement  $\lambda(f)<+\infty$ is \textit{independent of the branches} of the function $y(\zeta)$. It follows from the Theorem \ref{T:exponent-of-convergence} that we must have $C=D=0$ and hence so are $A=B=0$. Hence the solution (\ref{E:f-solution}) is expressed in the form

\begin{equation}\label{E:lommel-sum}
f(z)=\e^{-Nz} \sum_{j=1}^n \sigma_j S_{\mu_j,\, \nu}(L\e^{Mz}).
\end{equation}\smallskip

To complete the proof of the Theorem \ref{T:Chiang-Yu}, we need to prove that when $\sigma_j$ is non-zero, then $\mu_j$ and $\nu$ must satisfy either

\begin{equation}\label{E:desired-equations}
\cos\bigl({\mu_j+\nu\over 2}\pi\bigr)=0\quad\textrm{or}\quad 1+\e^{(-\mu_j+\nu)\pi i}=0,
\end{equation}\smallskip

\noindent where $j \in \{1,\,2,\ldots,\, n\}$. Following a similar idea as in \cite[p.  145]{Chiang:Yu:2008}, we have from the Remark \ref{R:Bessel-lommel-independence-branches} that $S_{\mu_1,\,\nu}(L\e^{Mz}),\,S_{\mu_2,\,\nu}(L\e^{Mz}),\ldots,\,S_{\mu_n,\,\nu}(L\e^{Mz})$ are entire functions in the $z$-plane and that each $S_{\mu_j,\,\nu}(L\e^{Mz})$ ($j=1,\,2,\ldots,\,n$) is \textit{independent of the branches} of $S_{\mu_j,\,\nu}(\zeta)$. We choose for a  $j \in \{1,\,2,\ldots,\,n\}$ such that $\sigma_j \neq 0$. So we can rewrite the solution (\ref{E:lommel-sum}) as

\begin{equation}\label{E:lommel-sum-rewrite}
f(z)=\sigma_j\e^{-Nz}S_{\mu_j,\,\nu}(L\e^{Mz}\e^{-m\pi i})+\sum_{k=1 \atop k \neq j}^n \sigma_k \e^{-Nz}S_{\mu_k,\,\nu}(L\e^{Mz}),
\end{equation}\smallskip

\noindent where the function $S_{\mu_j,\,\nu}(\zeta)$ belongs to the branch $-(m+1)\pi < \arg \zeta <-(m-1)\pi$ and the other Lommel functions $S_{\mu_1,\,\nu}(\zeta),\,\ldots,\,S_{\mu_{j-1},\,\nu}(\zeta),\,S_{\mu_{j+1},\,\nu}(\zeta),\,
\ldots,\,S_{\mu_n,\,\nu}(\zeta)$ are in the principal branch $-\pi <\arg \zeta<\pi$ and $m$ is an arbitrary but otherwise fixed non-zero integer.\smallskip

\begin{remark}\label{R:4.1}\rm
We note again that in the following discussion that we only consider the case $\mu_j-\nu=-2p_j-1$. The other case $\mu_j+\nu=-2p_j-1$ can be dealt with similarly  by applying the property that each $S_{\mu_j,\,\nu}(\zeta)$ is an even function of $\nu$.
\end{remark}\smallskip

Suppose that $\mu_j-\nu=-2p_j-1$ for some non-negative integer $p_j=p$. If $-\nu \not\in \{0,\,1,\,2,\ldots\}$ or $\nu=0$, then it follows from the Lemma \ref{L:lommel-continuation-general-2}(a) and (b) that the solution (\ref{E:lommel-sum-rewrite}) can be expressed in the form (\ref{E:f-solution-hankel}) with

\begin{equation*}
C=\left\{
    \begin{array}{ll}
      \D \frac{\sigma_j (-1)^pK_+'}{2^{2p}p!(1-\nu)_p}, & \hbox{if $-\nu \not\in \{0,\,1,\,2,\ldots\}$;} \\
      \D \frac{\sigma_j(-1)^pK_+''}{2^{2p}(p!)^2}, & \hbox{if $\nu=0$,}
    \end{array}
  \right.
\end{equation*}
and
\begin{equation*}
D=\left\{
    \begin{array}{ll}
      \D \frac{\sigma_j (-1)^pK_-'}{2^{2p}p!(1-\nu)_p}, & \hbox{if $-\nu \not\in \{0,\,1,\,2,\ldots\}$;} \\
      \D \frac{\sigma_j(-1)^pK_-''}{2^{2p}(p!)^2}, & \hbox{if $\nu=0$,}
    \end{array}
  \right.
\end{equation*}\smallskip

\noindent where $K_\pm'$ and $K_\pm''$ are the constants defined in (\ref{E:Ks-definition}). In order to apply the Theorem \ref{T:exponent-of-convergence}, we may follow closely the argument used in \cite[Proposition 4.4 (i) and (ii)]{Chiang:Yu:2008}, where if we have $C \neq 0$ or $D \neq 0$ for any integer $m$, then we will obtain a contradiction to the free choice of the integer $m$. Hence $\lambda(f)=+\infty$ and then either $C$ or $D$ must be zero.  This implies that (\ref{E:desired-equations}) holds, as required.\smallskip

If $\nu=-n$ for a positive integer $n$, then it follows from Lemma \ref{L:lommel-continuation-general-2}(c) that the solution (\ref{E:lommel-sum-rewrite}) (with $m$ replaced by $2m$) is given by
\begin{align*}
f(z)&=\frac{(-1)^{n+p}\sigma_j\e^{-Nz}}{2^{2p+n}n!(p!)^2(1+n)_p}(L\e^{Mz})^{-n}\\
&\qquad\times\bigg\{B_n(L\e^{Mz}) \Big[K_+''H_0^{(1)}(L\e^{Mz})+K_-''H_0^{(2)}(L\e^{Mz})\Big]\bigg. \notag\\
&\qquad -(L\e^{Mz})C_n(L\e^{Mz}) \Big[K_+''H_1^{(1)}(L\e^{Mz})+K_-''H_1^{(2)}(L\e^{Mz})\Big]\bigg\} \notag\\
&\quad+\sum_{j=1}^n \sigma_j \e^{-Nz}S_{\mu_j,\,\nu}(L\e^{Mz}). \notag
\end{align*}
It is obvious that the above expression is \textit{not} in the form (\ref{E:f-solution-hankel}) so that the Theorem \ref{T:exponent-of-convergence} does not apply in this case. In order to find an alternative approach to show  $\lambda(f)=+\infty$, we show that the function $h(\zeta)$ defined by:

\begin{align}\label{E:h-definition}
h(\zeta)&:=\zeta^{-n}\Big\{B_n(\zeta)\big[K_+''H_0^{(1)}(\zeta)+K_-''H_0^{(2)}(\zeta)\big] \\
&\qquad-\zeta C_n(\zeta)\big[K_+''H_1^{(1)}(\zeta)+K_-''H_1^{(2)}(\zeta)\big]\Big\}+\sum_{j=1}^n\sigma_j S_{\mu_j,\,\nu}(\zeta) \notag
\end{align} has \textit{infinitely many} zeros in the principal branch of $H_0^{(1)}(\zeta),\,H_0^{(2)}(\zeta),\,
H_1^{(1)}(\zeta),\,H_1^{(2)}(\zeta)$ and $S_{\mu_j,\,\nu}(\zeta)$.
Therefore we suppose that $-\pi<\arg \zeta<\pi$. Then the asymptotic expansions (\ref{E:hankel1-asy}), (\ref{E:hankel2-asy}) and (\ref{E:lommel-asymptotic-expansion}) with putting $p=1$ yield

\begin{align}\label{E:h-hat}
\widehat{h}(\zeta)&:=\bigg(\frac{\pi\zeta}{2}\bigg)^{\frac{1}{2}}h(\zeta) \notag\\
&=\zeta^{-n}\bigg[D_n^+(\zeta)K_+''\e^{-i\frac{\pi}{4}}\e^{i\zeta}+
D_n^-(\zeta)K_-''\e^{i\frac{\pi}{4}}\e^{-i\zeta}+D_n^+(\zeta)K_+''\e^{-i\frac{\pi}{4}}\e^{i\zeta}O(\zeta^{-1}) \\
&\qquad+D_n^-(\zeta)K_-''\e^{i\frac{\pi}{4}}\e^{-i\zeta}O(\zeta^{-1})\bigg]+\sum_{j=1}^n\widehat{\sigma}_j \zeta^{\mu_j-\frac{1}{2}}\big[1+O(\zeta^{-1})\big], \notag
\end{align}\smallskip

\noindent where $D_n^\pm(\zeta)=B_n(\zeta) \pm i\zeta C_n(\zeta)$. To find the number of zeros of $h(\zeta)$ in $-\pi<\arg \zeta<\pi$, we need the following result: \smallskip

\begin{lemma}\label{L:Dn-degree} Suppose that $n$ is a positive integer. Then at least one of  $D_n^+(\zeta)$ or $D_n^-(\zeta)$ has degree $n$.
\end{lemma}

\begin{proof}[Proof of Lemma \ref{L:Dn-degree}] We let

\begin{equation*}
B_n(\zeta):=\sum_{j=0}^n \beta_j \zeta^j \quad \mbox{and}\quad C_n(\zeta):=\sum_{j=0}^n \gamma_j \zeta^j,
\end{equation*}\smallskip

\noindent where $\beta_1,\ldots,\,\beta_n,\,\gamma_1,\ldots,\, \gamma_n$ are complex constants. We prove the lemma by induction on $n$. When $n=1$, we have $D_1^\pm(\zeta)=\pm i\zeta$, so the statement is true. Assume that it is also true when $n=k$ for a positive integer $k$. Without loss of generality, we may assume that $\deg D_k^+(\zeta)=k$ so that

\begin{equation}\label{E:inductive-assumption}
\beta_k+i\gamma_k \neq 0.
\end{equation}\smallskip

When $n=k+1$, it follows from the recurrence relations (\ref{E:recurrence-A_n-B_n-C_n}) for $B_n(\zeta)$ and $C_n(\zeta)$ that

\begin{equation*}
D_{k+1}^\pm(\zeta)=-2kD_k^\pm(\zeta)+\zeta B_k'(\zeta) \pm i\zeta^2 C_k'(\zeta) \pm i\zeta D_k^\pm(\zeta).
\end{equation*}\smallskip

\noindent It is easy to check that the coefficients of the $\zeta^{k+1}$ in $D_k^\pm(\zeta)$ are given by
$\pm ik\gamma_k \pm i(\beta_k\pm i\gamma_k)$ respectively. If $\deg D_{k+1}^\pm (\zeta) \le k$, then we have $\pm ik\gamma_k\pm i(\beta_k\pm i\gamma_k)=0$ so that both $\gamma_k$ and $\beta_k$ are zero which certainly contradict to our inductive assumption (\ref{E:inductive-assumption}). Hence we must have $\deg D_{k+1}^+(\zeta)=k+1$ or $\deg D_{k+1}^-(\zeta)=k+1$, completing the proof of the lemma.
\end{proof}\smallskip

We can complete the proof of the theorem now.\newline

{We recall that we have assumed $\mu_j-\nu=-2p_j-1$ for some non-negative integer $p_j=p$ and $\nu=-n$ for a positive integer $n$, see the paragraphs following the Remark \ref{R:4.1}.} By the Lemma \ref{L:Dn-degree}, we may suppose that $D_n^\pm (\zeta)=C_\pm \zeta^n+\cdots$ and $\widehat{C}_\pm=C_\pm\e^{\mp i\frac{\pi}{4}}$, where $C_+ \neq 0$. Then the expression (\ref{E:h-hat}) induces

\begin{align*}
\widehat{h}(\zeta)&=\widehat{C}_+K_+''\e^{i\zeta}\big[1+O(\zeta^{-1})\big]
+\widehat{C}_-K_-''\e^{-i\zeta}\big[1+O(\zeta^{-1})\big] +\sum_{j=1}^n\widehat{\sigma}_j \zeta^{\mu_j-\frac{1}{2}}\big[1+O(\zeta^{-1})\big]
\end{align*}\smallskip

\noindent which is in the form (\ref{E:Modify-soln}) with $\widehat{C}$ and $\widehat{D}$ replaced by $\widehat{C}_+K_+''$ and $\widehat{C}_-K_-''$ respectively. Therefore the proof of the Theorem \ref{T:exponent-of-convergence} can be applied without change to show that if at least one of $\widehat{C}_+K_+'' \neq 0$ or $\widehat{C}_-K_-'' \neq 0$, then the function $h(\zeta)$ has infinitely many zeros in $-\pi<\arg \zeta<\pi$ and thus $\lambda(f)=+\infty$, a contradiction. Hence we conclude that $\mu_j-\nu$ cannot be an odd negative integer.\smallskip

Now we can apply the analytic continuation formula in the Lemma \ref{L:lommel-continuation-general-1} with this {fixed} integer $m$ to get

\begin{align}\label{E:lommel-continuation-generalpi}
S_{\mu_j,\, \nu}(L{\rm e}^{M z}{\rm e}^{-m\pi
i})&=K_+P_m(\cos\nu\pi,{\rm e}^{-\mu_j\pi i})H_{\nu}^{(1)}(L{\rm e}^{Mz})\\
&\quad+K_+{\rm e}^{-\nu\pi i}P_{m-1}(\cos\nu\pi,{\rm e}^{-\mu_j\pi i})H_{\nu}^{(2)}(L{\rm e}^{Mz}) \notag\\
&\quad +(-1)^m{\rm e}^{-m\mu_j \pi i}S_{\mu_j,\, \nu}(L{\rm e}^{Mz}),\notag
\end{align}\smallskip

\noindent where $P_m(\cos\nu\pi,{\rm e}^{-\mu_j\pi i})$ is the polynomial as defined in the Lemma \ref{L:lommel-continuation-general-1}. Then the expressions (\ref{E:lommel-sum-rewrite}) and (\ref{E:lommel-continuation-generalpi}) give

\begin{align}\label{E:lommel-sum-modify-further}
f(z)&=K_+\sigma_j{\rm e}^{-Nz} P_m(\cos\nu\pi,{\rm e}^{-\mu_j\pi i})H_{\nu}^{(1)}(L{\rm e}^{Mz})\\
&\quad+K_+\sigma_j{\rm e}^{-Nz} {\rm e}^{-\nu\pi i}P_{m-1}(\cos\nu\pi,{\rm e}^{-\mu_j\pi i})H_{\nu}^{(2)}(L{\rm
e}^{Mz})\notag\\
&\quad +(-1)^m\sigma_j{\rm e}^{-m\mu_j \pi i-Nz}S_{\mu_j,\, \nu}(L{\rm e}^{Mz})+\sum_{k=1 \atop{k\not=j}}^n \sigma_k {\rm e}^{-Nz}S_{\mu_k,\, \nu}(L {\rm e}^{M z}).\notag
\end{align}\smallskip

If either of the coefficients of $H_{\nu}^{(1)}(L{\rm e}^{Mz})$ and $H_{\nu}^{(2)}(L{\rm e}^{Mz})$ in the (\ref{E:lommel-sum-modify-further}) is non-zero, then the Theorem \ref{T:exponent-of-convergence} again implies that $\lambda(f)=+\infty$ which is impossible. Thus we must have

\begin{equation}\label{E:subnormal-zero-anypi}
K_+P_m(\cos\nu\pi, {\rm e}^{-\mu_j\pi i})=0=K_+{\rm e}^{-\nu\pi i}P_{m-1}(\cos\nu\pi,{\rm
e}^{-\mu_j\pi i}).
\end{equation}\smallskip

Now we are ready to derive the equations (\ref{E:desired-equations}), we again recall that the value of
$\lambda(S_{\mu_j,\, \nu}(L {\rm e}^{M z}))$ must be independent of branches of the function $S_{\mu_j,\, \nu}(\zeta)$ which is equivalent to equations (\ref{E:subnormal-zero-anypi}) hold for \textit{each} integer $m$. It is clear from the Lemma \ref{L:lommel-continuation-general-1}(b) that $P_m(\cos \nu\pi,\,{\rm e}^{-\mu_j \pi i})\equiv 0$ and $P_{m-1}(\cos \nu\pi,\,{\rm e}^{-\mu_j \pi i})\equiv 0$ do not hold simultaneously for any integer $m$. Thus $K_+=0$ must hold, \textit{i.e.}, when $\sigma_j\not=0$,

\begin{equation}\tag{6.30}
\cos\bigl({\mu_j+\nu\over 2}\pi\bigr)=0\quad\textrm{or}\quad 1+{\rm
e}^{(-\mu_j+\nu)\pi i}=0,
\end{equation}\smallskip

\noindent where $j \in \{1,\,2,\ldots,\, n\}$. Hence we deduce the first and the second conditions in (\ref{E:conditions}) from the first and the second equations in (\ref{E:desired-equations}) respectively. A detailed deduction can be found in \cite[pp. 154, 155]{Chiang:Yu:2008}. This completes the proof of the theorem.\smallskip

\section{\bf A proof of Corollary \ref{T:Chiang-Yu-Lommel-Zeros}}

If $n=1$, then the assumption gives $\sigma_1 \neq 0$ so that $F(\zeta)=\sigma_1S_{\mu_1,\,\nu}(\zeta) \not \equiv 0$. If $n \ge 2$, then it follows from the Lemma \ref{L:lommel-independence} that $F(\zeta) \not \equiv 0$. Thus the function $F(\zeta)$ as defined in (\ref{E:F-definition}) is non-trivial so that we may suppose that the function (\ref{E:F-definition}) has finitely many zeros in every branch of $\zeta$. Then the entire function

\begin{equation*}
f(z)=F(\e^z)=\sum_{j=1}^n \sigma_j S_{\mu_j,\,\nu}(\e^z)
\end{equation*}\smallskip

\noindent is certainly a solution of the equation (\ref{E:Chiang-Yu}) with $L=M=1$ and $\lambda(f)<+\infty$. Hence the Theorem \ref{T:Chiang-Yu} implies that either $\mu_j+\nu=2p_j+1$ or $\mu_j-\nu=2p_j+1$ for non-negative integers $p_j$, where $j=1,\,2,\ldots,\,n$.\smallskip

Conversely, if either $\mu_j+\nu=2p_j+1$ or $\mu_j-\nu=2p_j+1$ for non-negative integers $p_j$, where $j=1,\,2,\ldots,\,n$, then the Remark \ref{R:lommel-terminate} shows that each $S_{\mu_j,\,\nu}(\zeta)/\zeta^{\mu_j-1}$ is a polynomial in $1/\zeta$ so that $S_{\mu_j,\,\nu}(\zeta)$ has only finitely many zeros in every branch of $\zeta$, where $j=1,\,2,\ldots,\,n$. Thus this implies that the function (\ref{E:F-definition}) has finitely many zeros in every branch of $\zeta$. This completes the proof of the Corollary \ref{T:Chiang-Yu-Lommel-Zeros}.\bigskip

\section{\bf Non-homogeneous {function-theoretic} quantization-type results}

The explicit representation and the zeros distribution of an entire solution $f(z)$ of either the equation

\begin{equation}
\label{E:Chiang-Ismail-1}
    y^{\prime\prime} +{\rm e}^z y=Ky
\end{equation}\smallskip

\noindent or the equation

\begin{equation}
\label{E:Chiang-Ismail-2}
    y^{\prime\prime}+\Big(-\frac14 {\rm e}^{-2z}+\frac12{\rm e}^{-z}\Big)y=Ky
\end{equation}\smallskip

\noindent were studied by Bank, Laine and Langley \cite{Bank:Laine:1982}, \cite{Bank:Laine:Langley:1986} (see also \cite{Shimomura:2002}). Later, Ismail and one of the authors strengthened  \cite{Chiang:Ismail:2006} (announced in \cite{Chiang:Ismail:2002}) their results. In fact, they discovered that the solutions of (\ref{E:Chiang-Ismail-1}) and (\ref{E:Chiang-Ismail-2}) can be solved in terms of Bessel functions and Coulomb Wave functions respectively. Besides, they identified that two classes of classical orthogonal polynomials (Bessel and generalized Bessel polynomials respectively)  appeared in the explicit representation of solutions under the \textit{boundary condition} that the exponent of convergence of the zeros of the solution $f(z)$ is finite, \textit{i.e.}, $\D{\lambda(f)=\lim_{r\to+\infty}\frac{\log n(r,\,\frac{1}{f})}{\log r}<+\infty}$. This also results in a complete determination of the eigenvalues and eigenfunctions of the equations. {We call such pheonomenon a \textit{function-theoretic quantization result} for the differential equations (\ref{E:Chiang-Ismail-1}) and (\ref{E:Chiang-Ismail-2})}.\smallskip

It is also well-known that both equations have important physical applications. For examples, the Eqn. (\ref{E:Chiang-Ismail-1}) is derived as a reduction of a non-linear Schr\"odinger equation in a recent study of Benjamin-Feir instability phenomenon in deep water in \cite{Segur:Henderson:Carter:Hammack:2005}, while the second Eqn. (\ref{E:Chiang-Ismail-2}) is {an exceptional case of} a standard classical diatomic model in quantum mechanics introduced by P. M. Morse in 1929 \cite{Morse:1929}\begin{footnote} {See \cite[pp. 1-4]{Slater:1969} for a historical background of the Morse potential.}\end{footnote} and is a basic model in the recent $\mathcal{PT}-$symmetric quantum mechanics research \cite{Znojil:1999} (see also \cite{Bender:2007}).\smallskip

In \cite[Theorem 6.1]{Chiang:Yu:2008}, the authors considered the following differential equation

\begin{equation}\label{E:reduced-Chiang-Yu}
f''+({\rm e}^z-K)f=\sigma 2^{\mu-1}{\rm e}^{\frac{1}{2}(\mu+1)z}
\end{equation}\smallskip

\noindent which is a special case of the equation (\ref{E:Chiang-Yu}) when $L=2,\,M=\frac{1}{2},\,N=0$ and $n=1$ in the Theorem \ref{T:Chiang-Yu}, where $\D K=\frac{\nu^2}{4}$. They obtained the necessary and sufficient condition on $K$ so that the equation (\ref{E:reduced-Chiang-Yu}) admits subnormal solutions which are related to classical polynomials and / or functions, \textit{i.e.}, Neumann's polynomials, Gegenbauer's {generalization of Neumann's} polynomials, Schl\"{a}fli's polynomials and Struve's functions. This exhibits a kind of {function-theoretic} quantization phenomenon for non-homogeneous equations.\smallskip

Now the following result holds trivially by our main Theorem \ref{T:Chiang-Yu}:

\begin{theorem}\label{T:quantizations} With each choice of parameters as indicated in Table \ref{mytable} below, we have a necessary and sufficient condition on $K$ that depends on the non-negative integer $p$ so that the equation {\rm (\ref{E:reduced-Chiang-Yu})} admits a solution with finite exponent of convergence of zeros. Furthermore, the forms of such solutions are given explicitly in Table \ref{mytable}:\newline

\begin{table}[h]\vspace*{-3ex}\extrarowheight=7pt
\caption[]{Special cases of {\rm (\ref{E:reduced-Chiang-Yu})}.} \label{mytable}
\vspace{0cm}
\begin{tabular}{cccc}
   \hline
   & Cases &  Corresponding $K$ & Solutions with finite exponent \\
   &       &                    & of convergence of zeros        \\ \hline
  (1)& $\mu=1$ & $p^2$ & $2\sigma {\rm e}^{\frac{z}{2}}O_{2p}(2{\rm e}^{\frac{z}{2}})$ \\
  (2)& $\mu=0$ & $\D\frac{(2p+1)^2}{4}$  & $\D \frac{2\sigma}{2p+1}{\rm e}^{\frac{z}{2}}O_{2p+1}(2{\rm e}^{\frac{z}{2}})$\\
  (3)& $\mu=-1$ & $(p+1)^2$   & $\D \frac{\sigma}{4(p+1)}S_{2p+2}(2{\rm e}^{\frac{z}{2}})$ \\
  (4)& $\mu=\nu$ &  $\D \frac{(2p+1)^2}{16}$ & $\sigma 2^{p-\frac{1}{2}}\sqrt{\pi}p!\Big[{\bf H}_{p+\frac{1}{2}}(2{\rm e}^{\frac{z}{2}})-Y_{p+\frac{1}{2}}(2{\rm e}^{\frac{z}{2}})\Big]$ \vspace{3pt}\\ \hline
\end{tabular}
\end{table}
\end{theorem}\smallskip

Here $O_{2p}(\zeta)$ and $O_{2p+1}(\zeta)$ are the Neumann polynomials of degrees $2p$ and $2p+1$ respectively; $S_p(\zeta)$
is the Schl\"{a}fli polynomial and ${\bf H}_{p+\frac{1}{2}}(\zeta)$ is the Struve function, see \cite[9.1, 9.3, 10.4]{Watson:1944}.\smallskip

\appendix
\section{\bf Preliminaries on Bessel and the Lommel functions}\label{A:Bessel}
\subsection{Bessel functions}
 Let $m$ be an integer. We record here the following analytic continuation formulae for the Bessel functions
\cite[3.62]{Watson:1944}:

\begin{align}
J_\nu(\zeta \e^{m\pi i})&= \e^{m\nu \pi i} J_\nu(\zeta),
\label{E:bessel-first-continuation}\\
Y_\nu(\zeta \e^{m\pi i})&= \e^{-m\nu \pi i} Y_\nu
(\zeta)+ 2i \sin(m\nu\pi) \cot(\nu\pi)
J_\nu(\zeta).\label{E:bessel-second-continuation}
\end{align}\smallskip

\noindent We recall the {\it Bessel functions of the third kind of
order} $\nu$ \cite[3.6]{Watson:1944} are given by

\begin{equation}\label{E:hankels-definition}
H_\nu^{(1)}(\zeta)=J_\nu(\zeta)+iY_\nu(\zeta), \quad
H_\nu^{(2)}(\zeta)=J_\nu(\zeta)-iY_\nu(\zeta).
\end{equation}\smallskip

\noindent They are also called the \textit{Hankel functions of order $\nu$ of the first and second kinds}. The asymptotic expansions of $H_\nu^{(1)}(\zeta)$ and
$H_\nu^{(2)}(\zeta)$ are also recorded as follows:

\begin{equation}\label{E:hankel1-asy}
\left(\frac{\pi
\zeta}{2}\right)^{\frac{1}{2}}H_\nu^{(1)}(\zeta)= \e^{i\left(\zeta-\frac{1}{2}\nu\pi-\frac{1}{4}\pi\right)}
\left[\sum_{k=0}^{p-1}\frac{\left(\frac{1}{2}-\nu\right)_k\left(\frac{1}{2}
+\nu\right)_k} {k!(2i\zeta)^k}+R_{p}^{(1)}(\zeta)\right]
\end{equation}\smallskip

\noindent where $R_{p}^{(1)}(\zeta)=O(\zeta^{-p})$ in $-\pi < \arg \zeta < 2\pi$;

\begin{equation}\label{E:hankel2-asy}
\left(\frac{\pi
\zeta}{2}\right)^{\frac{1}{2}}H_\nu^{(2)}(\zeta)= \e^{-i\left(\zeta-\frac{1}{2}\nu\pi-\frac{1}{4}\pi\right)}
\left[\sum_{k=0}^{p-1}\frac{\left(\frac{1}{2}-\nu\right)_k\left(\frac{1}{2}
+\nu\right)_k} {k!(-2i\zeta)^k}+R_{p}^{(2)}(\zeta)\right]
\end{equation}\smallskip

\noindent where $R_{p}^{(2)}(\zeta)=O(\zeta^{-p})$ in $-2\pi < \arg \zeta < \pi$. See \cite[7.2]{Watson:1944}.\smallskip

\subsection{An asymptotic expansion of $\lommel(\zeta)$ and linear independence of Lommel's functions}\label{A:asymptotic-expansion}
It is known that when $\mu \pm \nu$ are not odd positive integers,
then $S_{\mu,\,\nu}(\zeta)$ has the asymptotic expansion

\begin{align}\label{E:lommel-asymptotic-expansion}
    S_{\mu,\,\nu}(\zeta)&=\zeta^{\mu-1}\left[\sum_{k=0}^{p-1} \frac{(-1)^k
    c_k}{\zeta^{2k}}\right] +O\left(\zeta^{\mu-2p}\right)
\end{align}\smallskip

\noindent for large $|\zeta|$ and $|\arg \zeta|<\pi$, where $p$ is a positive integer. See also \cite[10.75]{Watson:1944}. As a result, we see that the
asymptotic expansions (\ref{E:hankel1-asy}), (\ref{E:hankel2-asy}) and (\ref{E:lommel-asymptotic-expansion}) are valid \textit{simultaneously} in the range $-\pi<\arg \zeta<\pi$.\smallskip

\begin{remark} \label{R:lommel-terminate}\rm It is clear that (\ref{E:lommel-asymptotic-expansion}) is a series in descending powers of $\zeta$ starting from the term $\zeta^{\mu-1}$ and (\ref{E:lommel-asymptotic-expansion}) terminates if one of the numbers $\mu \pm \nu$ is an odd positive integer. In particular, if $\mu-\nu=2p+1$ for some non-negative integer $p$, then we have $K_+=0$ in the analytic continuation formula (\ref{E:lommel-continuation-general}) and thus, in this degenerate case, the formula (\ref{E:lommel-continuation-general}) becomes $S_{2p+1+\nu,\,\nu}(\zeta \e^{-m\pi i})=\e^{-m\nu\pi i} S_{2p+1+\nu,\,\nu}(\zeta)$ for every integer $m$ and $|\arg \zeta|<\pi$.
\end{remark}

The following concerns about the linear independence of the Lommel functions $S_{\mu_j,\,\nu}(\zeta)$.

\begin{lemma}\label{L:lommel-independence} {\cite[Lemma 3.12]{Chiang:Yu:2008}} Suppose $n \ge 2$, and $\mu_j$ and $\nu$ be complex numbers such that $\Re(\mu_j)$ are all distinct for $j=1,2,\ldots, n$. Then the Lommel functions $S_{\mu_1,\, \nu}(\zeta),\,S_{\mu_2,\,
\nu}(\zeta),\ldots$, $S_{\mu_n,\, \nu}(\zeta)$ are linearly independent.
\end{lemma}

\end{document}